\documentclass[10pt]{amsart}
\usepackage{amsmath,amsthm,amssymb, amscd}
\usepackage{graphicx}
\usepackage[margin=1.5in]{geometry}
\usepackage{flafter}
\usepackage{ mathrsfs }
\usepackage{enumitem}
\usepackage{lipsum}
\usepackage{cite}
\usepackage{bbold}
\usepackage[all,cmtip]{xy}
\usepackage{float}
\usepackage[usenames]{color} 
\usepackage[utf8]{inputenc} 
\floatstyle{boxed}
\restylefloat{figure}
\theoremstyle{plain}
\newtheorem{thm}{Theorem}[section]
\newtheorem{prop}[thm]{Proposition}
\newtheorem{cor}[thm]{Corollary}
\newtheorem{lem}[thm]{Lemma}

\theoremstyle{definition}
\newtheorem{defn}[thm]{Definition}
\newtheorem{exmp}[thm]{Example}
\theoremstyle{remark}
\newtheorem*{rem}{Remark}

\newtheorem*{ack}{Acknowledgments}
\numberwithin{equation}{section}

\newcommand{\A}{\mathcal{A}}
\newcommand{\C}{\mathbb{C}}

\newcommand{\hen}{\mathcal{X}_{\acute{e}t}}
\newcommand{\Perf}{\text{Perf}}
\newcommand{\Map}{\text{Map}}
\newcommand{\Catperf}{\text{Cat}^{\text{perf}}}
\newcommand{\Cat}{\text{Cat}}
\newcommand{\shvcat}{ShvCat^{\acute{e}t}(X)}
\newcommand{\Bet}{\mathbf{Betti}_{X}}
\newcommand{\Sa}{\mathcal{S}}
\newcommand{\Sm}{\text{Sm}}
\newcommand{\Mod}{\text{Mod}}
\newcommand{\et}{\acute{e}t}
\newcommand{\walpha}{\widetilde{\alpha}}
\newcommand{\Pic}{\text{Pic}}
\newcommand{\Aa}{\mathbb{A}^1} 
\newcommand{\Alg}{\text{Alg}}
\newcommand{\Pre}{\text{Pre}}

\newcommand{\Prs}{\mathcal{P}r^{L,st, \omega}}

\newcommand{\Aff}{\text{Aff}_{\mathbb{C}}}
\begin{document}
\title{Derived Azumaya Algebras and twisted K-theory}
\author{Tasos Moulinos}
\address{Department of Mathematics, Statistics and Computer Science, University of Illinois at Chicago, 851 S. Morgan
Street, Chicago, IL, 60607-7045, USA}
\email{tmouli2@uic.edu}
\begin{abstract}
We construct a relative version of topological $K$-theory, in the sense of \cite{blanc}, over an arbitrary quasi-compact, quasi-separated $\C$-scheme $X$.  This has as input a $\text{Perf}(X)$-linear stable $\infty$-category and output a sheaf of spectra on $X(\C)$, the space of complex points of $X$.  We then characterize the values of this functor on inputs of the form $\Perf(X, A)$, for $A$ a derived Azumaya algebra over $X$.  In such cases we show that this coincides with the $\alpha$-twisted topological $K$-theory of $X(\C)$ for some appropriately defined twist of $K$-theory.  We use this to provide a topological analogue of a classical result of Quillen's on the algebraic $K$-theory of Severi-Brauer varieties.

\end{abstract}
\maketitle
\section{Introduction}

In the conjectural theory of motives as envisaged by Grothendieck, one deals with realization functors into various abelian categories of modules or representations. One can then take such categories, which are themselves realized as hearts of certain $t$-structures on the derived category level, as rough approximations to the hypothetical abelian category of mixed motives. Given a scheme $X$ over a characteristic zero field, one recovers, for example, the singular cohomology groups of the associated complex analytic space $X(\C)$. Similarly, one has  \'{e}tale and $\ell$-adic realization functors, which exist in arbitrary characteristic.   It has since become important to view  motives and therefore their realizations in families; this allows one to keep track of various aspects of the geometry of the underlying scheme. One ultimately recovers not just an abelian group or a representation, but a sheaf of such objects on an associated topological space or $\infty$-topos. In this vein, Ayoub defines, in \cite{ayoub}, a notion of Betti realization over an arbitrary base complex scheme. 

In the ``noncommutative" setting of algebraic geometry, one replaces a scheme with its associated triangulated category (rather, its ``enhancement" as a dg or stable $\infty$-category) of perfect complexes. The focus then shifts onto dg-categories themselves as the main object of study; this allows for a somewhat more unified framework to study objects from various settings such as symplectic geometry or representation theory. The aforementioned realization functors will now give rise to realization functors out of derived, noncommutative versions of motives. One is able to moreover recover a surprising amount of data about the underlying geometrical or arithmetic nature of the objects being studied. An example of this machinery at work can be found in \cite{blanctoen}, where the authors apply a noncommutative extension of the  $\ell$-adic realization functor to the dg-category of matrix factorizations of a Landau-Ginzburg model $(X,f)$.

It is quite useful, in this categorified setting, to have realization functors for non-commutative motives parametrized by various schemes or more geometric objects. With this in mind, we focus our study on a noncommutative version of topological realization for complex schemes, the theory of which was set forth by Blanc in \cite{blanc}. 
\subsection{Topological $K$-theory of complex dg categories}

In \cite{blanc}, the author constructs an invariant of dg-categories over the complex numbers $K^{top} : \Catperf(\C) \to Sp$ which he describes as a ``noncommutative analog of rational Betti cohomology for algeraic varieties". For any dg-category $T$, the topological $K$-theory spectrum $K^{top}(T)$  he constructs will be a module over the periodic complex $K$-theory spectrum. Perhaps rather surprisingly, the following equivalence holds for any scheme $X \in \text{Sch}_{\C}$, separated of finite type:
\[
K^{top}(\text{Perf}(X)) \simeq KU(X(\C)).
\]
Here $\text{Perf}(X)$ denotes dg-category of perfect complexes of $\mathcal{O}_{X}$-modules and $KU(X(\C))$ is  topological $K$-theory of the associated space of complex points $X(\C)$.\\

The first main goal of this work is to define a relative version of topological $K$-theory. 
Namely  for every scheme $X \in \text{Sch}_{\C}$,we construct a functor
\[
K^{top}_{X}: \text{Cat}^{\text{perf}}(X) \to Shv_{Sp}( X(\C))
\]
taking values in the $\infty$-category of sheaves of spectra on $X(\C)$ (as opposed to just spectra)  We refer to this as the \emph{relative topological $K$-theory of $X$}.  When $X = Spec(\C)$, we show that this reproduces the original version of topological $K$-theory due to Blanc.

\subsection{Azumaya algebras and the twisted derived category}
Our next objective is to use this relative version of topological $K$-theory to study certain ``twists" of  the derived category over a scheme. To motivate what we mean by this, we begin with the rather classical notion of a central simple algebra over a field $k$. These noncommutative $k$-algebras, upon extending scalars to a splitting field of $k$, become equivalent to matrix algebras. In his seminal paper \cite{grothaz}, Grothendieck globalized this into the notion of an Azumaya algebra of a scheme. In short, an Azumaya algebra over a scheme $X$ will be a sheaf of algebras which upon pulling back to an \'{e}tale cover, is equivalent to a sheaf of matrix algebras. Every Azumaya algebra over a scheme moreover gives rise to a canonical torsion class 
\[
\alpha \in H^{2}_{\et}(X, \mathbb{G}_m)
\]
in \'{e}tale cohomology. 
Grothendieck asked if every torsion class $\alpha \in H^{2}_{\et}(X, \mathbb{G}_m)$ is realizable by an Azumaya algebra. This question has since been given a negative answer. However, it was the insight of To\"{e}n that the analogous question has a positive answer in the setting of derived algebraic geometry. In \cite{toenaz}, he studied a notion of a local system of dg categories over a derived scheme. These have a precise cohomological characterization; Grothendieck's question in this setting becomes a question of whether or not these local systems of dg-categories have global generators. As it turns out this is true in this case; see \cite{antieaugepner} for the version in the setting of spectral algebraic geometry.  
Derived Azumaya algebras over derived schemes are studied up to a certain $\infty$-categorical notion of Morita equivalence (defined globally, however, over a derived scheme or stack) and  the \emph{derived Brauer group}, $\pi_{0} \mathbf{Br}(X)$, is the group of derived Azumaya algebras up to this notion of equivalence.  

\subsection{Twisted K-theory; algebraic and topological}

Fix a topological space $X$. One may define a vector bundle $E \xrightarrow{\pi} X$ by specifiying vector bundles $E_{i}$ over an open cover $\{U_{i} \}_{i \in I}$ equipped with isomorphisms
\[
f_{ij}: E_{i} \to E_{j}
\]
on $U_{i} \cap U_{j}$ satisfying the condition that $ f_{ij} \circ f_{jk} = f_{ik}$.

Now fix a class $\alpha$ in  $H^{3}(X, \mathbb{Z})$. One can choose a $\text{GL}_{1}(\C)$-valued cocycle $\{\alpha_{ijk}\}$ (with respect to a fixed cover of $X$) that realizes this class in cohomology.  A twisted bundle is given by modifying the standard gluing data by this cocycle  so that 
\[
f_{ij} f_{jk} f_{ik}^{-1} = \alpha_{ijk} \in \text{GL}_{1}(\C).
\]

The collection of \emph{ $\alpha$-twisted} vector bundles may be assembled into a category $\text{Vect}_{\alpha}(X)$. This category comes equipped with a suitable notion of an exact sequence, and so we may take its Grothendieck group and obtain a rudimentary version of the $\alpha$-\emph{twisted topological $K$-theory} of $X$.  This notion of $K$-theory was first introduced by Donovan and Karoubi in \cite{donovan} wherein they defined a  local system in $H^{3}(X,\mathbb{Z}) \times H^{1}(X, \mathbb{Z}/2) \times BBSU_{\otimes}(X)$ and then further developed by Rosenberg, Atiyah and others.  See, for instance, \cite{atiyah}.  Often times, the twisted form of $K$-theory has a geometric interpretation.  In particular if $X$ is a space equipped with a bundle of projective spaces, one may define a form of twisted $K$-theory as the global sections of a certain bundle of $KU$-module spectra twisted by the action of $PU_{n}$ on the fibers.  It is well known that $\pi_{0}$ of this spectrum coincides with the Grothendieck group construction applied to the exact category of $\alpha$-twisted bundles on $X$ for the resulting twist $\alpha$.
\\

There is an obvious notion of twisted algebraic $K$-theory as well.  This is  obtained by taking the $K$-theory of  $\text{Perf}(X, \alpha)$, the category of perfect complexes of $\alpha$-twisted sheaves of $\mathcal{O}_{X}$-modules, equivalently, modules over the associated derived Azumaya algebra.

\subsection{Statement of Results}
Our first application of relative topological $K$-theory we construct identifies the topological $K$-theory of complexes of $\alpha$-twisted sheaves on $X$ with the twisted $K$ theory of $X(\C)$.     

\begin{thm}
Let $X$ denote a quasi-compact, quasi-separated scheme over the complex numbers.  Let $\alpha \in \pi_{0}\mathbf{Br}_{0}(X)$ be a Brauer class, with $\emph{Perf}(X, \alpha) \in \emph{Cat}^{\emph{perf}}(X)$ the associated $\emph{Perf}(X)$-linear category. Then there exists a functorial equivalence
\[
K^{top}_{X}( \emph{Perf}(X, \alpha) ) \simeq \underline{KU^{\widetilde{\alpha}}(X (\C))}.
\]
Here, $ \underline{KU^{\widetilde{\alpha}}(X (\C))}$ is the local system of invertible $KU$-modules associated to a twist $\widetilde{\alpha}: X(\C) \to \text{Pic}_{KU}$ obtained functorially from $\alpha$.  
\end{thm}

As a corollary we  obtain the following result on the ``absolute" topological $K$-theory.   

\begin{cor}
Let $X \in \text{Sch}_{\C}$ be a quasi-compact, quasi separated scheme and let  $\alpha \in H^2_{\acute{e}t}(X, \mathbb{G}_{m})$ be a torsion class in \'{e}tale cohomology corresponding to an ordinary Azumaya algebra over $X$.  Then 
\[
K^{top}( \emph{Perf}(X, \alpha)) \simeq KU^{\walpha}(X(\C))
\]
where $\walpha \in H^{3}(X, \mathbb{Z})$ is the class in singular cohomology obtained  via the topological realization functor
\[
H^2_{\acute{e}t}(X, \mathbb{G}_{m}) = [X, B^{2} \mathbb{G}_{m}] \xrightarrow{|| - ||} [X(\C), ||B^{2}(\mathbb{G}_{m})|| ] \simeq  [X(\C), B^{2}(S^1) ] = H^{3}(X(\C), \mathbb{Z}). 
\]
\end{cor}

\noindent We then display an application of these results purely in the realm of topology.  For this, we let $X$ be a quasi compact scheme, and let $P \to X$ denote a Severi Brauer-scheme of relative dimension $n-1$ over $X$.  This will mean that $P$ is, \'{e}tale locally  on $X$, equivalent to projective space $\mathbb{P}_{X}^{n-1}$.  As is well known, there is a sheaf of Azumaya algebras  $A$, over $X$ canonically associated to $P$. It is a classical theorem, due to Quillen in \cite{quillen}  that 
\[
K(P) \simeq K(X) \oplus K(A^{\otimes{1}}) \oplus ... \oplus K(A^{\otimes n-1}).
\]
In particular, this means that the $K$-groups of perfect complexes of $\mathcal{O}_{P}$-modules decompose as a direct sum of the $K$-groups of the categories of perfect complexes of $A^{\otimes{n}}$ modules where $A^{\otimes n}$ is the $n$-th tensor product of $A$.  \\

One might wonder whether the analogue of this result is  true in the topological setting.  More precisely, if $X$ is a topological space, and if  $\pi: P \to X$ is projective fiber bundle, does there exist a decomposition of the topological $K$-theory of $P$ into summands involving the twisted topological $K$-theory of the base space?  Using our methods, we prove the following affirmative result, as a generalization of the Leray-Hirsch theorem to this context: 

\begin{thm}
Let $X$ be a finite CW-complex.  Let $\pi: P \to X$ be a bundle of rank $n-1$ projective spaces classified by a map $ \alpha: X \to BPGL_{n}(\C)$.  Let $\tilde{\alpha}: X \to B^{2} \C ^{\times} $ be the composition of this map along the map $ BPGL_{n}(\C) \to B^{2} \C ^{\times} \simeq K( \mathbb{Z}, 3).$  This gives rise to an element $\widetilde{\alpha} \in H^3( X, \mathbb{Z})$.  Then the topological $K$-theory of the total space $P$ decomposes as follows:

\[
KU^{*}(P) \simeq KU^{*}(X) \oplus KU^{\tilde{\alpha}}(X) ... \oplus KU^{\tilde{\alpha}^{n-1}}(X).
\]    
where $KU^{\widetilde{\alpha ^k}}(X)$ denotes the twisted $K$-theory with respect to the class $\widetilde{\alpha }^{k} \in H^{3}(X, \mathbb{Z})$.

\end{thm}  
This topological analogue of Quillen's $K$-theoretic result has been hitherto unknown.  
\\

We briefly outline the contents of this paper.  We begin in Section \ref{section2} by giving some background on presentable $\infty$-categories, $\infty$-topoi, and $\mathcal{E}$-linear $\infty$-categories and review how dg-categories can be subsumed in the $\infty$-categorical framework.  In Section 3 we review the notion of a $\mathcal{C}$-valued sheaf for arbitrary $\mathcal{C}$ which we repeatedly need.  In Section 4 we introduce topological $K$-theory of dg categories, introduced by Blanc in \cite{blanc}.  In Section 5 we review the theory of derived Azumaya algebras and the Brauer stack.  In  Section 6, we use the \'{e}tale local triviality of derived Azumaya algebras to show that their \'{e}tale sheafified $K$-theory is an invertible object in the symmetric monoidal $\infty$-category of \'{e}tale sheaves of spectra over $X$.  In Section 7 we introduce our definition of topological $K$-theory over an arbitrary base scheme which we refer to as \emph{relative topological $K$-theory}. In Section 8  we  give a basic overview of twisted topological $K$-theory as needed to state and prove our main theorem.  In Section 9 we state and prove our main theorem displaying the topological $K$-theory of an derived Azumaya algebra    as an invertible local system of $KU$-modules on $X(\C)$.  Finally in Section 10 we state and prove our theorem on the decomposition of the complex $K$-theory of a bundle of projective spaces over a finite CW-complex $X$.   

\begin{ack}
The author would like to express his deepest gratitude to his thesis advisors, Benjamin Antieau and Brooke Shipley, for generously sharing their ideas and for their guidance, advice, and support over the years.

\end{ack}
\section{Higher categorical preliminaries} \label{section2}

We utilize the language of $\infty$-categories.  We assume the reader is acquainted with the basic ideas and definitions, but will review the notions that are central to the execution of this work.  For a more thorough introduction to the theory one should consult \cite{lurie2009, lurie2016}

\subsection{Presentable $\infty$-categories and $\infty$-topoi} 
We recall the definition of presentable $\infty$-categories.  The categories we work with will often be either presentable,  or will arise as subcategories of compact objects of  presentable $\infty$-categories.  Our arguments and constructions will frequently depend on this structure being present in the situation at hand.

\begin{defn}
Let $\mathcal{C}$ be an $\infty$-category and $\kappa$ be an infinite regular cardinal.  We may form the Ind-category Ind$_{\kappa}(\mathcal{C})$ which is the formal completion of $\mathcal{C}$ under $\kappa$-filtered colimits.  We say an $\infty$ category $\mathcal{D}$ is \emph{accessible} if there exists a small $\infty$-category $\mathcal{C}$ and some regular cardinal $\kappa$ such that $\mathcal{D} \simeq Ind_{\kappa}( \mathcal{C})$.  If $\mathcal{C}$ is accessible and has all small colimits, then it is \emph{presentable}    
\end{defn}

The collection of all presentable $\infty$-categories can be organized into an $\infty$-category Pr$^{\text{L}}$ with morphisms consisting of those functors which are left adjoints.  This category is closed monoidal so that the $\infty$-category of functors between any two presentable $\infty$-categories,  Fun$^{\text{L}}(A, B)$, is itself presentable.  One may also organize the collection of presentable $\infty$-categories together with \emph{right adjoint functors}; we denote the resulting $\infty$-category by  $\mathcal{P}r^{\text{R}}$.  According to \cite[Corollary 5.5.3.4]{lurie2009}, these two categories are anti-equivalent to each other.  Furthermore, this identification gives the following useful description of the symmetric monoidal tensor product in $\mathcal{P}r^{L}$:
\[
\mathcal{C} \otimes^{\text{L}} \mathcal{D} \simeq Fun^{R}(\ \mathcal{C}^{op}, \mathcal{D})
\]
 
We now give a characterization of $\infty$-topoi among presentable $\infty$-categories as those that satisfy a certain form of \emph{descent}.  A concise way to describe this is as follows, taken from \cite{umkehr}:

\begin{defn}
Let $\mathcal{X}$ be a presentable $\infty$-category. Let $\mathcal{O}_{\mathcal{X}} := Fun( \Delta^{1}, \mathcal{X})$ denote the arrow category of $\mathcal{X}$ and let $  \mathcal{O}_{\mathcal{X}} \to  \mathcal{X}  $ be functor sending a morphism to its target.  This has the property of being a \emph{Cartesian fibration}, in particular giving us a functor 
\[
\mathcal{X}^{op} \to \widehat{Cat_{\infty}}
\]

\noindent We say that $\mathcal{X}$ is an $\infty$-topos if the following descent condition holds: if $T \simeq colim_{\alpha}T_{\alpha}$ is a colimit in $\mathcal{X}$ then there is an induced map 
\[
\mathcal{X}_{/T} \to \text{lim}_{\alpha} \mathcal{X}_{/T{\alpha}}
\]
which is an equivalence.

\end{defn}
This definition in fact characterizes $\infty$-topoi uniquely amongst locally Cartesian closed presentable $\infty$-categories. 

\begin{defn}
Given two $\infty$-topoi $\mathcal{X}$ and $\mathcal{Y}$, we define  a \emph{geometric morphism}  $f: \mathcal{X} \to \mathcal{Y}$ to be an adjoint pair of functors
\[
f^{*}: \mathcal{Y} \rightleftarrows \mathcal{X} : f_{*}
\]
such that the left adjoint preserves finite limits.  
\end{defn}

For the purposes of this paper we may think of $\infty$-topoi as categories of sheaves on an $\infty$-categorical version of a site equipped with a Grothendieck topology. We will not be using too many intricacies of higher topos theory; we invite the reader from now to keep in mind the following specific examples that will arise: $Shv(X)$, the category of sheaves (of spaces) on a topological space $X$, or $Shv^{\acute{e}t}(X)$, the category of sheaves on the site of (smooth) schemes over a scheme $X$, endowed with the \'{e}tale topology.   
\\

\noindent We will need the definition of a locally constant object in an $\infty$-topos:    

\begin{defn} \label{bimbo}
Let $\mathcal{X}$ denote an arbitrary $\infty$-topos.  Let $\mathcal{F}$ be an object of $\mathcal{X}$.  We say that $\mathcal{F}$ is \textit{constant} if it lies in the essential image of the unique geometric morphism $\pi^{*}: \mathcal{S} \to \mathcal{X}$.  We will say that $\mathcal{F}$ is \emph{locally constant} if there exists a small collection of objects $\{U_{\alpha} \subseteq X\} $ such that the following conditions are satisfied: 
\begin{enumerate}
\item The objects $U_{\alpha}$ cover $\mathcal{X}$; that is, there is an effective epimorphism $\sqcup U_{\alpha} \to \mathbf{1}$ where $\mathbf{1}$ denotes the final object of $\mathcal{X}$. 
\item For each $\alpha \in S$, the product $\mathcal{F} \times U_{\alpha}$  is a constant object of the $\infty$-topos $\mathcal{X}/_{U_{\alpha}}$.  
\end{enumerate}
\end{defn}

\subsection{Symmetric monoidal $\infty$-categories}

We do not give a complete account of the theory of symmetric monoidal $\infty$-categories as we will only use a small part of the theory.  The inquisitive reader may find more in \cite[Chapter 2]{luriedescent}.  

Let $\Gamma$ denote the category of pointed finite sets, with morphisms pointed maps of finite sets.  An $\infty$-operad is then an $\infty$-category $\mathcal{O}^{\otimes}$ and a functor
\[
p: \mathcal{O}^{\otimes} \to N(\Gamma)^{\otimes}
\]
satisfying the conditions of \cite[Definition 2.1.1.10]{lurie2016}.  A symmetric monoidal $\infty$-category is then an $\infty$-operad  $\mathcal{C}^{\otimes}$ together  with a \emph{cocartesian fibration  of operads}\cite[2.1.2.18]{lurie2016}   

\[
p : \mathcal{C}^{\otimes} \to N(\Gamma)^{\otimes}
\]

\noindent such that

\begin{itemize}
\item for each $n \geq 0$, the associated functors $\mathcal{C}^{\otimes}_{[n]} \to \mathcal{C} ^{\otimes}_{[1]}$ determine an equivalence of $\infty$-categories 
$\mathcal{C}_{[n]} \simeq  \mathcal{C}^{n}_{[1]}$, where $\mathcal{C}_{[n]} = p^{-1}([n])$.  

\end{itemize}

\begin{rem}
We may reverse engineer this definition in order to obtain functors 
\[
\mathcal{C} \times \mathcal{C} \to \mathcal{C},
\]
reminiscent of the 1-categorical notion of a symmetric monoidal category.
\end{rem}

\begin{rem}
If $\mathcal{C}$ is a symmetric monoidal $\infty$-category,
then its homotopy category will be symmetric monoidal in the 1-categorical sense. 
\end{rem} 

\begin{defn}
Let $\mathcal{C}^{\otimes}$ be a symmetric monoidal  $\infty$-category and let $\mathcal{O}^{\otimes}$ be an $\infty$-operad. We define an $\mathcal{O}^{\otimes}$-algebra object of $\mathcal{C}$ to be 
a map of operads $\alpha: \mathcal{O}^{\otimes} \to \mathcal{C}^{\otimes}$(\cite[Definition 2.1.2.7]{lurie2016})
We let $\Alg_{\mathcal{O}}(\mathcal{C})$ denote the subcategory of $\infty$-category $Fun_{N(\Gamma)^{\otimes}}(\mathcal{O}^{\otimes}, \mathcal{C}^{\otimes})$ of functors over $N(\Gamma)^{\otimes}$ spanned by the $\infty$-operad maps. If $\mathcal{O}^{\otimes} = N(\Gamma)^{\otimes}$, then we will denote $\Alg_{N(\Gamma)}(\mathcal{C})$ by C$\Alg(\mathcal{C})$ and refer to this as the $\infty$-category of commutative algebra object of $\mathcal{C}$.
\end{defn}

We conclude this section by recalling the notion of a dualizable object in a symmetric monoidal $\infty$-category:

\begin{defn}
Let $\mathcal{C}^{\otimes}$ be a symmetric monoidal $\infty$-category and let $ A$ be an object of $\mathcal{C}$.  We say $A$ is \emph{dualizable} if  there exists an object $A^{\vee} $ together with a evaluation map $\epsilon: A \otimes A^{\vee} \to 1$ and a coevaluation map $\eta: 1 \to A^{\vee} \otimes A $ such that the composition:
\[
A \simeq A \otimes 1 \xrightarrow{1 \otimes \eta } A \otimes A^{\vee} \otimes A \xrightarrow{\epsilon \otimes 1} 1 \otimes A \simeq A 
\]
and 
\[
A^{\vee} \simeq 1 \otimes A^{\vee}  \xrightarrow{ \eta \otimes 1} A^{\vee} \otimes A \otimes A^{\vee} \xrightarrow{ 1 \otimes  \epsilon } A^{\vee} \otimes 1 \simeq  A^{\vee}
\]
is equivalent to the identity  \\
We say $A$ is \emph{invertible} if the evaluation and coevaluation morphisms are  moreover equivalences.  
\end{defn}   
 
\subsection{Stable $\infty$-categories} \label{stable}
Stable $\infty$-categories provide a natural framework for dealing with structures arising in algebraic geometry, algebraic topology and representation theory. In  addition, they serve as natural inputs for the algebraic $K$-theory functor.  The reader should consult \cite{blumberg} for a characterization of algebraic $K$-theory functor as an invariant of stable $\infty$-categories satisfying a certain universal property.  We briefly recall their definition.

\begin{defn}
An $\infty$-category is \emph{stable} if it has finite limits and colimits and if fiber sequences and cofiber sequences coincide.  A functor $F: \mathcal{C} \to \mathcal{D}$ between two stable $\infty$-categories is \emph{exact} if it preserves fiber sequences.
\end{defn}

\begin{defn}
The collection of all small stable, idempotent complete $\infty$-categories and exact functors can itself be organized into an $\infty$-category, which we denote by $\text{Cat}^{\text{perf}}$.  Similarly, the collection of all stable presentable $\infty$-categories, with morphisms left adjoint functors, can be also organized into an $\infty$-category, denoted by $\mathcal{P}r^{L, st}$.
\end{defn}

We remark that, as in \cite[Proposition 5.5.7.10]{lurie2009} there is an ind-completion functor 
 \[
Ind: Cat^{\text{perf}} \to \mathcal{P}r^{L,st}
\]
sending a small stable $\infty$-category $\mathcal{C}$ to $Ind(\mathcal{C})$, its formal cocompletion under $\omega$-filtered colimits which will be a stable presentable category. This induces an equivalence (\cite[Section 3.1]{blumberg}) between $\text{Cat}^{\text{perf}}$ and $\Prs$, the subcategory of  $\mathcal{P}r^{L,st}$ consisting of the compactly-generated stable presentable $\infty$-categories with morphisms left adjoint functors preserving compact object.  We implictly use this identification throughout the course of this paper. 

\begin{rem}
The $\infty$-category $\mathcal{P}r^{L, st}$ is symmetric monoidal with unit the $\infty$-category of spectra.  This symmetric monoidal structure is characterized by the property that maps out of the tensor product $A \otimes B$ correspond to bifunctors out of the  product $A \times B$ preserving colimits in each variable. As discussed in \cite[Section 4]{blumberg2}, the $\infty$-category $\Prs$ inherits a symmetric monoidal structure from  $\mathcal{P}r^{L,st}$; via the equivalence described above, there will symmetric monoidal structure on $\Catperf$, with unit $Sp^{\omega}$.  
\end{rem}

\subsection{$\mathcal{E}$-linear $\infty$-categories}{\label{linearcat}}

\begin{defn}
Let $\mathcal{E} \in$ CAlg$(\Catperf)$ be a small symmetric monoidal stable $\infty$-category.  We set $\text{Cat}^{\text{perf}}(\mathcal{E}) := \text{Mod}_{\mathcal{E}}(\text{Cat}^{\text{perf}})$.   Similarly, we may write \\ $\Cat(\mathcal{E}) := \text{Mod}_{Ind(\mathcal{E})}(\mathcal{P}r^{L,st}).$  We refer to objects of this category as $\mathcal{E}$-\emph{linear categories}; these are stable infinity categories $\mathcal{C}$ equipped with an exact functor $\mathcal{C} \otimes \mathcal{E} \to \mathcal{C}$. 
 
\end{defn}

\begin{exmp}
As an example,  we take the $\infty$-category of \emph{compact} $R$-modules $\mathcal{E}= \Perf(R)$ for $R$ any $\mathbb{E}_{\infty}$ ring spectrum.  We refer to these as $R$-linear $\infty$-categories and use the standard convention $\Catperf(R) := \Catperf(\Perf(R))$  

\end{exmp}

\begin{exmp}
Fix a quasi-compact, quasi-separated scheme  $X$. In this situation we set $\mathcal{E} = \Perf(X)$ the symmetric monoidal $\infty$-category of perfect complexes of $\mathcal{O}_{X}$-modules.  We let $\text{QCoh}(X) \simeq Ind(\Perf(X))$; this will be compactly generated by the perfect complexes. We set $\Catperf(X) := \Catperf(\Perf(X)).$
\end{exmp}

\begin{rem}
For the following, we will actually need the assumption that $\mathcal{E}$ is \emph{rigid}, so that all objects are dualizable.  This is for instance satisfied in the above examples, which encompasses the situations we will deal with.  Hence, we assume once and for all that $\mathcal{E}$ is rigid.   
\end{rem}

Let $\mathcal{C} \in \Catperf(\mathcal{E})$.  Then, for every object $a \in \mathcal{C}$, we may define a  functor $\mathcal{E} \to \mathcal{C}$ sending $e \to e \otimes a$.  This functor preserves finite colimits and therefore admits an Ind-right adjoint, $\mathcal{C}^{\mathcal{E}}(a, -) : \mathcal{C} \to Ind(\mathcal{E})$.  This gives $\mathcal{C}$ the structure of an $Ind(\mathcal{E})$-enriched category.  For the example above, with $\mathcal{C}$ an $R$-linear category, this means the mapping object, $\mathcal{C}(a,b)$ has the structure of an $R$-module for every object $a, b \in \mathcal{C}$.

We will often be concerned with $\mathcal{E}$-linear $\infty$-categories with the following properties which we now define:

\begin{defn}
Let $\mathcal{E} \in \text{CAlg}(\Catperf)$ and let $\mathcal{C} \in Cat^{\text{perf}}(\mathcal{E})$   We say $\mathcal{C}$ is
\begin{itemize}
\item \emph{smooth} if $\mathcal{C}$ is compact as  an object of $Ind( \mathcal{C}^{op} \otimes_{\mathcal{E}} \mathcal{C})$ 
\item \emph{proper} if for all objects $a, b$, the mapping object $\mathcal{C}(a, b) \in \mathcal{E}$ is compact as an object of $Ind(\mathcal{E})$
\item \emph{saturated} if it is both smooth and proper.   
\end{itemize}

\end{defn}

We have the following characterization of saturated $\mathcal{E}$-linear $\infty$-categories, originally found in \cite{blumberg} whose proof may be found in \cite{hoyois}.  

\begin{thm}
Let $\mathcal{C} \in \emph{Cat}^{\emph{perf}}(\mathcal{E})$.  Then $\mathcal{C}$ is saturated if and only if it is dualizable. 
\end{thm}

\subsection{DG-categories} \label{dg}
We assume the reader is familiar with the basics of dg-categories.  Therefore we just describe how, for our purposes, they can be subsumed into the language of stable $\infty$-categories. For a fixed commutative ring $k$ there exists an $\infty$-category $Dg(k)$ encoding the homotopy theory of (small) dg-categories up to quasi-equivalences.  See \cite{keller} for a more thorough discussion.  There exists a ``dg-nerve" functor $N_{dg}: Dg(k) \to Cat_{\infty}$.  Moreover, a dg category $\mathcal{C}$ will be presentable (in the d.g. category sense) if and only if  the associated $\infty$-category $N_{dg}( \mathcal{C})$ is presentable.  Hence we can define the restriction $N_{dg} : Dg(k)^{lp} \to \mathcal{P}r^{L,st}$, where $Dg(k)^{lp}$ is the $\infty$-category of presentable dg-categories and colimit preserving dg-functors.  This functor is known to be conservative and reflects fully faithfullness.  It also  preserves the associated homotopy categories and therefore preserves the notion of compact generators.  One can therefore restrict even further to obtain $N_{dg}: Dg(k)^{cc} \to \mathcal{P}r^{L, st, \omega}$ where $Dg(k)^{cc}$ denotes compactly generated presentable dg categories.    
 \\
Now, $Dg(k)^{cc} \simeq Dg(k)^{idem}$, the presentable $\infty$-category of small dg-categories up to Morita equivalence.  In \cite{cohn2013}, it is shown that  there is a factorization 
\[
Dg(k)^{cc} \to Mod_{Mod_{k}}(\mathcal{P}r^{L, st, \omega}) \to \mathcal{P}r^{L, st, \omega}
\]
with the last map being the forgetful functor and the first map being an equivalence.  Hence, we are able to identify Cat$^{\text{perf}}(k)  \simeq Mod_{Mod_{k}} (\mathcal{P}r^{L, st, \omega})$  with the $\infty$-category of small idempotent complete $k$-dg categories. \\

This identification will allow us to work with dg-categories and categories of their invariants (for instance, their $K$-theory spectra) in the same setting.

\section{ $\mathcal{C}$-valued Sheaf categories} \label{sheaf}

We will be dealing with categories of sheaves valued in categories other than spaces.   We collect here some basic properties of these categories of sheaves.  

\begin{defn}
Let $\mathcal{X}$ be an $\infty$-topos and $\mathcal{C}$ be an arbitrary $\infty$-category containing all small limits.  We define the category of $\mathcal{C}$ valued sheaves on $\mathcal{X}$ to be  $Shv_{\mathcal{C}}(\mathcal{X}) := Fun^{lim}( \mathcal{X}^{op}, \mathcal{C})$, the category of limit-preserving functors from $\mathcal{X}^{op} $ to $\mathcal{C}$.  Note that if $\mathcal{C}$ is presentable, this can be alternatively described as  $Fun^{R}( \mathcal{X}^{op}, \mathcal{C})$, the internal hom in $\mathcal{P}r^{R}$ the $\infty$-category of presentable $\infty$-categories, together with right adjoints.  Furthermore, by \cite[Proposition 4.8.1.17]{lurie2016} this is equivalent to the tensor product $\mathcal{X} \otimes^{L} \mathcal{C}$ in $\mathcal{P}r^{L}$.  
  
 \end{defn} 

\begin{exmp}
As an example let $\mathcal{X} = Shv(X)$, the $\infty$-topos of sheaves on a topological space $X$ and let $\mathcal{C} = Sp$, the $\infty$-category of spectra.  Then $Shv_{Sp}(X) $ is the category of sheaves on the space $X$, valued in spectra.   
\end{exmp}
\begin{exmp}
Let $\mathcal{X}$ be an arbitrary $\infty$-topos and let $\mathcal{C} = \text{Cat}^{\text{perf}}$ be the category of small stable, idempotent complete $\infty$ categories.  This category itself is presentable (see \cite{blumberg}) and we can make sense of $Shv_{\text{Cat}^{\text{perf}}}(\mathcal{X})$ as $Fun^{R}(\mathcal{X}^{op}, \text{Cat}^{\text{perf}})$.
\end{exmp}

Let $\mathcal{T}$ be small $\infty$-category equipped with a Grothendieck topology, i.e. a Grothendieck topology on its underlying homotopy category.  We abuse notation (justifiably) and denote $Shv_{\mathcal{C}}(Shv(\mathcal{T}))$ by $Shv_{\mathcal{C}}(\mathcal{T})$ Then the inclusion $ Shv_{\mathcal{C}}( \mathcal{T}) \hookrightarrow \text{Pre}_{\mathcal{C}}(\mathcal{T})$ admits a left adjoint; as we are in the context of presentable $\infty$-categories, the appropriate theory of Bousfield localization applies, displaying sheafification as a localization functor.  Moreover, assuming $\mathcal{C}$ is symmetric monoidal, the morphisms which we invert by may be chosen so that this this is a symmetric monoidal functor; this gives us a natural symmetric monoidal structure on  $\text{Shv}_{\mathcal{C}}(\mathcal{T})$.  \\

We now define a locally constant object, in analogy with the definition in the setting of $\infty$-topoi.  

\begin{defn}
Let $\mathcal{X}$ be an arbitrary $\infty$-topos and let $\mathcal{C}$ be a presentable $\infty$-category.  We say first $\mathcal{F} \in \text{Shv}_{\mathcal{C}}(\mathcal{X})$ is \emph{constant} if it lies in the image of the morphism $\mathcal{C} \to \text{Shv}_{\mathcal{C}}(\mathcal{X})$ induced by the terminal geometric morphism.

We say $\mathcal{F}$ is \emph{locally constant} if there is a collection of objects $\{U_{\alpha} \} \in \mathcal{X}$ such that there is an effective epimorphism $\cup U_{\alpha} \to \mathbf{1}$ in $\mathcal{X}$ so that for every $\alpha \in S$, $\phi^{*}(\mathcal{F})$ is constant in  $\text{Shv}_{\mathcal{C}}(\mathcal{X}_{/U_{\alpha}})$.

\end{defn}

\section{Topological K-theory of $\C$-linear $\infty$-categories}
We recall in this section the definition of topological $K$-theory and recall some of its properties.  Most of this is can be found in \cite{blanc}.  

\subsection{Topological realization}
Recall the functor 
\[
X \mapsto X(\C)
\]
associating to a finite type scheme over the complex numbers its set of complex points.  Every  set $X(\C) := Map_{\text{Sch}/_{\C}}(Spec(\C), X)$ will come endowed with a natural topology, allowing us to view this as a functor landing in topological spaces.  We can then compose this with the singular functor $\text{Sing}(-): Top \to \mathcal{S}$ to obtain a simplicial set.  As in \cite{blanc}, we let $\Aff$ denote the category of affine $\C$-schemes of finite type.  This gives the following commutative triangle:
\[
\xymatrix{
\text{Aff}_{\C} \ar@{^{(}->}[d] \ar[r] &
\mathcal{S}  \\
\text{Pre(Aff}_{\C}) \ar[ru] }
\]

\noindent where the vertical arrow is the Yoneda embedding $\mathcal{Y}: \text{Aff}_{\C} \to \Pre(\text{Aff}_{\C})$ and the diagonal arrow is the left Kan extension of the ``complex points functor".  As this is a functor of presentable $\infty$-categories, we can take its stabilization to obtain the following functor:
\[
|| - ||_{\mathbb{S}} : \text{Pre}_{Sp}(\text{Aff}_{\C}) \to Sp
\]   
 which we refer to as spectral realization. Note here that, by \cite[Remark 1.4.2.9]{lurie2016} the stabilization $\text{Stab}(\Pre(\Aff))$, of  $\Pre(\Aff))$, is equivalent to  $\Pre_{Sp}(\Aff)$. Hence, given any presheaf of spectra $\mathcal{F}: \text{Aff}_{\C}^{op} \to Sp$, there is a spectrum $|| \mathcal{F}||_{\mathbb{S}}$ functorially associated to it.   

\subsection{Topological K-theory}

\begin{defn}
Let $T \in \Catperf(\C)$ be a $\C$-linear stable $\infty$-category, which by Section \ref{dg}, we think of as a dg-category over $\C$.  To $T$ we associate a presheaf of spectra $\underline{K(T)}: \text{Aff}_{\C} \to Sp$ sending 
\[
\text{Spec}(A) \mapsto K(T \otimes_{\C} \Perf(A)),
\]    
the $K$-theory of the category $T \otimes_{\C} \Perf(A)$ where ``$\otimes_{\C}$" is the tensor product in Cat$^{\text{perf}}(\C)$ described in Section \ref{stable}.  We define  the \emph{semi-topological k-theory}, $K^{st}(T) : = || \underline{K(T)} ||_{\mathbb{S}}$ to be the realization of this presheaf.        
\end{defn} 

It is shown in \cite{blanc} that $K^{st}(\mathbb{1}) = ku$ the connective $K$-theory spectrum.  Since $|| - ||_{\mathbb{S}}$ is symmetric monoidal, it follows that $K^{st}(T)$ is a $ku$-module for any dg-category $T$.  Hence, there exists an action by the Bott element $\beta \in \pi_{2}(ku)$.  The topological $K$-theory of $T$ is defined as follows:

\begin{defn}
Let $T \in \text{Cat}^{\text{perf}}(\C)$ and let $K^{st}(T)$ be semi-topological $K$-theory of $T$, as defined above.  The  \emph{topological k-theory} of $T$ is defined to be 
\[
K^{top}(T) := L_{KU}(K^{st}(T)) \simeq K^{st}(T) \wedge_{ku} KU,
\]
the inversion of $K^{st}(T)$ with respect to the Bott element.   
\end{defn}

We collect here several structural properties about topological $K$-theory.

\begin{thm}
Let $K^{top}: \emph{Cat}^{\emph{perf}}(\C) \to Sp$ be the functor of topological $K$-theory.  Then, 

\begin{enumerate}[label=\alph*.,start=1]
    \item $K^{top}(-)$ commutes with filtered colimits, is Morita  invariant and is a localizing invariant in that it sends exact sequences of  dg-categories $\mathcal{A} \to \mathcal{B} \to \mathcal{C}$ to fiber sequences of spectra.  
    \item $K^{top}$ is lax symmetric monoidal.  
    \item If $X \xrightarrow{\phi} Spec(\C)$ is a separated scheme of finite type, then there exists a functorial equivalence
\[
K^{top}(\emph{Perf}(X)) \simeq KU^{*}(X(\C)).
\]
Here, $\emph{Perf}(X)$ denotes the dg (or $\C$-linear $\infty$-category) of perfect complexes of quasi-coherent-$\mathcal{O}_{X}$ modules on $X$.
\end{enumerate}

\end{thm}

\begin{proof}
This is \cite[Theorem 1.1]{blanc}.
\end{proof}

\section{Derived Azumaya Algebras and the Brauer stack} \label{tikka}

To state our main result we first give a brief account of the theory of derived Azumaya algebras due to To\"{e}n in \cite{toenaz} and Antieau-Gepner in \cite{antieaugepner}. We fix $R$ to be an arbitrary base connective $\mathbb{E}_{\infty}$-ring.  For our purposes $R$ will be  the Eilenberg-Maclane spectrum $H \C$ or a commutative algebra over $H \C$.

\begin{defn}
An $R$ algebra $A$ is an Azumaya $R$-algebra if $A$ is a compact generator of $Mod_{R}$  and if the natural $R$-algebra map
\[
A \otimes_{R} A^{op} \to End_{R}(A)
\]
is an equivalence.  
\end{defn}

Since $End_{R}(A)$ is Morita equivalent to the base ring, this means that if $A$ is an Azumaya algebra, then there exists some $A^{op}$ such that  $A \otimes_{R} A^{op}$ is Morita  equivalent to $R$.  This gives rise to the following characterization of Azumaya algebras.  See \cite{antieaugepner} for a proof.  

\begin{prop}
Let $\mathcal{C}$ be a compactly generated $R$-linear category.  Then $C$ is invertible in $Cat_{R, \omega}:= \emph{Mod}_{\emph{Mod}_R}(\Prs) \simeq \emph{Cat}^{\emph{perf}}(R)$ if and only if $\mathcal{C}$ is equivalent to $\emph{Mod}_{A}$ for an Azumaya R-Algebra A.

\end{prop}
 
Azumaya algebras were classically defined with the property that they are equivalent  to matrix algebras \'{e}tale-locally.  In the derived setting this can be stated as follows:

\begin{thm} \label{localtriviality}
If $A$ is a derived Azumaya algebra over a connective $E_{\infty}$ ring $R$ there exists a faithfully flat \'{e}tale $R$-algebra $S$ with the property that $A \otimes_{R} S$ is Morita equivalent to $S$.  
\end{thm} 

\begin{proof}
See for instance, \cite[Theorem 5.11]{antieaugepner}.
\end{proof}

We define the \emph{Brauer space} of $R$ to be $\text{Pic}(Cat_{R, \omega})$. However, we will give a ``global" and hence more general definition of the Brauer space, where the input is that of an \'{e}tale $\infty$-sheaf, as opposed to a commutative ring. For this we work in an $\infty$-category of sheaves on the category $\text{CAlg}_{R}^{cn}(Sp)$ of connective commutative algebras over $R$ endowed with a derived version of the \'{e}tale topology. We denote this category by $\text{Shv}^{\acute{e}t}_{R}$.  See \cite[section 2]{antieaugepner} for more details on this.  
Now let $\mathcal{C}at_R: \text{Aff}_{R}^{op} \to   Cat_{\infty}$ be the functor sending a commutative $R$-algebra $S$ to $Cat_{S}$.  Let $Cat^{desc}_{S} \subset Cat_{S}$ be the subcategory of $S$-linear categories satisfying \'{e}tale hyperdescent.  From this, we may define the subfunctor  $\mathcal{C}at_{R,}^{desc} \subset \mathcal{C}at_{R}$  of $R$-linear categories with $R$ hyper-descent, such that $ S \mapsto Cat_{S}^{desc}$.  By the arguments of \cite[Theorem 7.5]{luriedescent}, this defines an \'{e}tale hyperstack on $Shv^{\acute{e}t}_{R}$.   

Let now $\mathcal{A}lg : Aff_{R}^{op} \to Pr^{L}$ be the functor sending $Spec(S) \mapsto \text{Alg}(\text{Mod}_{S})$; by \cite{lurie2016} this defines a hyperstack on $Shv^{\acute{e}t}_{R}$.  There exists an obvious morphism of stacks $\mathcal{A}lg \to \mathcal{C}at^{desc}$ induced by the functor $ A \to Mod_{A}$ sending an algebra to its associated category of modules.  Let $\mathcal{A}z$ be the subfunctor of $\mathcal{A}lg$ corresponding to taking Azumaya Algebras.  
  
Finally, let $\mathbf{Az}, \mathbf{Alg}$ and $\mathbf{Pr}$ be the underlying $\infty$-sheaves (restrict to the maximal subgroupoid sectionwise).  We define the Brauer stack as follows:

\begin{defn}
Let $\mathbf{Az} \to \mathbf{Pr}$ be the induced map of $\infty$-sheaves defined above.  We define the \emph{Brauer stack} to be the \'{e}tale sheafification of the image of this map.  Given an \'{e}tale sheaf $X \in   Shv^{\acute{e}t}_{R}$, we let $\mathbf{Br}(X) := Map_{Shv^{\acute{e}t}_{R}}(X, \mathbf{Br})$ denote the \emph{Brauer space} of $X$.  Finally, we write $Br(X) := \pi_{0}(\mathbf{Br}(X)$ as the \emph{Brauer group} of $X$.  

\end{defn}

If $X =Spec(S)$ for $S$ a connective $E_{\infty}$-$R$-algebra then it is easy to see $\mathbf{Br}(X) \simeq Pic(Cat_{S, \omega})$ so we see that this definition agrees with the  aforementioned one .  

The Brauer stack may be viewed as a delooping of the Picard stack.  Indeed, since Azumaya algebras are \'{e}tale locally equivalent to the ground ring,    it follows that $\mathbf{Br}$ is a connected sheaf; moreover it is equivalent to the classifying space of the trivial Brauer class, which is the map $\mathcal{M}od: Spec(R) \to \mathbf{Br} \to \mathbf{Pr}$ sending $Spec(S) \mapsto \text{Mod}_{S}$.  The sheaf of auto-equivalences  of $\mathcal{M}$od is presicely the sheaf of line bundles.  Hence,
\[
\Omega \mathbf{Br} \simeq \mathbf{Pic}.
\]

This identification allows us to better understand the homotopy of $\mathbf{Br}$.  Recall the following split fiber sequence of \'{e}tale sheaves: 

\[
\mathbf{BGL}_{1} \to \mathbf{Pic}  \to \mathbb{Z}.
\]
This may be delooped to obtain the fiber sequence $\mathbf{B}^{2}\mathbf{GL}_{1} \to \mathbf{Br} \to \mathbf{B} \mathbb{Z}$.  This gives the following decomposition, for an \'{e}tale sheaf $X$, of the derived Brauer group 
\[
\pi_{0} \mathbf{Br}(X) \simeq H^{2}_{\acute{e}t}(X, \mathbb{G}_{m}) \times H^{1}_{\acute{e}t}(X, \mathbb{Z}).
\]       

\begin{rem}
Let $X$ be a quasi-compact, quasi-separated scheme.  Let $\alpha \in Br(X) = \pi_{0}(\mathbf{Br}(X))$ be a class in the Brauer group represented by a map $ \alpha : X \to \mathbf{Br}$ in $Shv^{\acute{e}t}_{\C}$.  By   \cite[Theorem 6.17]{antieaugepner}, this map will lift to a map $\alpha: X  \to \mathbf{Az}$ classifying a derived Azumaya algebra over $X$.  We may therefore think of $ H^{2}_{\acute{e}t}(X, \mathbb{G}_{m}) \times H^{1}_{\acute{e}t}(X, \mathbb{Z})$ as a Morita equivalence classes of derived Azumaya algebras.  In particular we will often be concerned with Brauer classes living in the \'{e}tale cohomology group $H^{2}_{\acute{e}t}(X, \mathbb{G}_{m})$.  

\end{rem}

\begin{defn} \label{twistedsheavescategory}
Let $X$ be a quasi-compact, quasi-separated scheme, and let $\alpha: X \to \mathbf{Br} \to \mathbf{Pr} $. We define, the $\infty$-category of $\alpha$-twisted $\mathcal{O}_{X}$-modules to be   
\[
\text{QCoh}(X, \alpha) := \text{holim}_{f:Spec(R) \downarrow X} \Mod_{\alpha \circ  f} 
\]
Note that for every $Spec(R) \to X$, we have $Mod_{\alpha \circ f} \in  \mathcal{P}r^{l, st}$. This limit exists and is computed in $\mathcal{P}r^{l, st}$. Finally, we set $\Perf(X, \alpha) := \text{QCoh}(X, \alpha)^{\omega}$, the  small subcategory of compact objects of $\text{QCoh}(X, \alpha)$.
\end{defn}

\section{Stacks of presentable $\infty$-categories \& \'{e}tale $K$-theory}
As announced in the introduction, we construct,  for any scheme, quasi-compact, quasi-separated over $\C$, a relative version of topological $K$-theory $K^{top}_{X}(T)$ with input  a $\text{Perf}(X)$-linear $\infty$-category and output a sheaf of $KU$-module spectra on $X(\C)$.  We will particularly be concerned with the values of this functor when $T$ is the $\infty$-category of perfect complexes  of $\mathcal{O}_{X}$-modules $\text{Perf}(X, \alpha)$, for $\alpha \in \pi_{0} \mathbf{Br}(X)$.  In this section, we describe an equivalent definition of $\mathbf{Br}(X)$, which will appear more naturally within our constructions.  We do this in the setting of sheaves of linear categories over the scheme $X$, a notion of independent interest.

Let $\Catperf: CAlg_{\C}\to \widehat{Cat_{\infty}}$ denote the functor sending 
$R \mapsto \text{Cat}^{\text{perf}}(R)$.  We may right-Kan extend this along the Yoneda embedding $\text{Aff}_{\C} \to Shv_{\C}^{\acute{e}t}$ to obtain the following commutative diagram of functors:

\[
\xymatrix{
\text{Aff}_{\C}^{op} \ar@{^{(}->}[d] \ar[r]^{\Catperf(-)  } &
\widehat{Cat_{\infty}}  \\
{Shv^{\acute{e}t}_{\C}}^{op}. \ar[ru]_{ShvCat^{\acute{e}t}(-)} }
\]
Hence, if $X$ is any scheme, or more generally any \'{e}tale sheaf we may informally describe an object $ \mathcal{C} \in \shvcat$ as an assignment, to any $Spec(S) \in \text{Aff}_{ X}$ an $S$-linear category $\mathcal{C}_{S} \in \text{Cat}^{\text{perf}}(S)$ together with equivalences
\[
\mathcal{C}_{S_1} \otimes_{\Perf(S_1)} \Perf(S_2) \simeq \mathcal{C}_{S_2}
\]
for any morphism of schemes $\phi: Spec(S_2) \to Spec(S_1)$ over $X$.  
        
By \cite[Theorem 1.5.2]{gaitsgory}, the functor $ShvCat^{\acute{e}t}(-)$ satisfies \'{e}tale hyperdescent on the category; hence 
\[
ShvCat^{\acute{e}t}(-): {Shv^{\acute{e}t}_{\C}}^{op} \to \widehat{Cat_\infty}
\]
is limit preserving.  

\begin{lem}
$\shvcat$ is a symmetric monoidal category for any $X \in Shv^{\acute{e}t}_{\C}$. 
\end{lem}

\begin{proof}
This follows from the properties of right Kan extension, together with the fact that $\widehat{Cat_{\infty}}$ is itself symmetric monoidal.
\end{proof}

We now investigate the Picard space of this category. 

\begin{prop} \label{picbr}
There is a natural equivalence of spaces $\emph{Pic}(\shvcat) \simeq \mathbf{Br}(X)$
for $X$ a scheme, quasi-compact and quasi-separated over $\C$.

\end{prop}

\begin{proof}
By \cite[Theorem 7.7]{umkehr}, the Picard space functor $\text{Pic}:  \widehat{Cat_{\infty}} \to \mathcal{S}$ itself preserves limits as it appears as a right adjoint in the adjoint pair  
\[
\text{Pre}: \text{CAlg}^{grp}(\Sa) \rightleftarrows \text{CAlg}(\Prs) : \Pic,
\]
where $Pre(Y) := Fun_{\infty}(Y, \mathcal{S})$.  Hence, the composition, $\text{Pic}(ShvCat^{\acute{e}t}(-))$ preserves limits, and therefore defines an \'{e}tale hypersheaf of spaces.  We display a map 
\begin{equation} \label{sheafmap}
\text{Pic}ShvCat^{\acute{e}t} \to \mathbf{Br}
\end{equation}
of objects in $Shv_{\C}^{\acute{e}t}$ and then argue that it is an equivalence.   There is a natural map of functors 
\begin{equation} \label{sheafmap2}
\Pic(\Catperf(-)) \to \mathbf{Br} \to \mathbf{Pr},
\end{equation}
essentially by definition of the Brauer sheaf in Section 5, together with the fact that an invertible object $\mathcal{C} \in  \Pic(\Catperf(R)) \simeq  \Pic(Cat_{R, \omega})$ corresponds to $Mod_{A}$ where $A$ is a derived Azumaya algebra over $R$.    This gives the desired map.  Note that the symmetric monoidal equivalence
\[
Ind: \Catperf \to \Prs 
\]

\noindent from section \ref{stable}  induces the equivalence $\Catperf_{\infty}(R) \simeq Cat_{R, \omega}$ . Furthermore, recall from section 5, that 
\[
\Pic(Cat_{R, \omega}) \simeq \mathbf{Br}(R)
\]
for every commutative $\C$-algebra $R$.   Let $\{\Gamma_{R}\}$ be the collection of functors $\Gamma_{R}: Shv^{\acute{e}t}_{\C} \to \mathcal{S}$  with $\Gamma_{R}(\mathcal{C}) = (\mathcal{C}_{R})$ for $R$ ranging over all commutative, connective $\C$ algebras.  This defines a conservative family of functors and therefore  detects equivalences in $Shv^{\acute{e}t}_{\C}$.  We conclude that the map in  \ref{sheafmap} is an equivalence of sheaves.  Hence we make the identification 
\[
\Pic(\shvcat) \simeq \mathbf{Br}(X).
\]
\end{proof}

For $X$, quasi-compact, quasi-separated, the $\infty$-category $\shvcat$ admits a perhaps simpler description, as follows.  Let $\text{Perf}(X)$ denote the $\infty$-category of perfect complexes of $\mathcal{O}_{X}$ modules.  This is well known to be a small, stable idempotent complete $\infty$-category.  In the setting of section \ref{linearcat} we set
\[
\Catperf(X) := \text{Mod}_{\text{Perf}(X)}(\Catperf_{\infty})     
\]

For $X$ a general \'{e}tale sheaf, it is not necessarily true that $\Catperf(X) \simeq \shvcat$.  This is true whenever $X$ is $1$-affine in the sense of \cite{gaitsgory}. 

\begin{thm}[Gaitsgory] \label{gaitsgory}
Let $X$ be a scheme.  Then there exists a symmetric monoidal adjunction
\[
\mathbf{Loc}: \Catperf(X) \rightleftarrows \shvcat : \mathbf{\Gamma_{X}}
\]
If $X$ is in particular quasi-compact, quasi-separated, then these are inverse equivalences and $\mathbf{\Gamma_{X}}$ is itself (strictly) symmetric monoidal. 
\end{thm}

\begin{proof}
This is \cite[Theorem 2.1.1]{gaitsgory}
\end{proof}
\begin{cor}
Let $X$ be a quasi-compact, quasi-separated $\C$-scheme. There is an equivalence of spaces
\[
\Pic(\emph{Cat}^{\emph{perf}}(X)) \simeq \mathbf{Br}(X).
\]
\end{cor}

\begin{proof}
By theorem \ref{gaitsgory}, $\Pic(\Catperf(X)) \simeq \Pic(\shvcat)$.  By proposition \ref{picbr}, 
\[
\Pic(\shvcat) \simeq \mathbf{Br}(X).
\]
The equivalence follows.   \end{proof}

\subsection{\'{E}tale sheafified $K$-theory} \label{yo}

Fix $X$ a quasi-compact, quasi-separated scheme over $Spec(\C)$.  Let $\text{Sm}_{X}$ denote the (nerve of) the category of schemes smooth and finite type over $X$, equipped with the \'{e}tale topology. We remind the reader that the classical \'{e}tale topology, which we now refer to agrees with the \'{e}tale topology  used in derived algebraic geometry (which we use in Section \ref{tikka}) in the sense that for any $E_{\infty}$ ring $A$, the category Calg$^{\acute{e}t}_{/A}$ spanned by \'{e}tale morphisms is equivalent to the nerve of the (ordinary) category of $\pi_{0}A$-algebras. This is \cite[Theorem 7.5.0.6] {lurie2016}  
Our definition of relative topological $K$-theory will factor through the $\infty$-category of \'{e}tale sheaves of spectra on this category which we now describe.    
\begin{defn}
Let $\text{Pre}(\Sm_{X}) := \text{Fun}( \text{Sm}_{X}^{op}, \mathcal{S}) $ denote the $\infty$-topos of presheaves of spaces on $X$.  We define $Shv^{\acute{e}t}(X) \subset \text{Pre}(\text{Sm}_{X}) $  to be the full subcategory  of hypersheaves with respect to the \'{e}tale topology.  \end{defn}

\begin{defn}
Using the notions from section \ref{sheaf} we may now define $Shv^{\acute{e}t}_{Sp}(X)$, the stable $\infty$-category of \'{e}tale hypersheaves of spectra on $\text{Sm}_{/X}$ to be
\[
Shv_{Sp}^{\acute{e}t}(X) := \text{Fun}^{lim}(Shv^{\acute{e}t}(X)^{op}, Sp) \simeq
Shv^{\acute{e}t}(X) \otimes^{L} Sp
\]
\end{defn}

The forgetful functor $Shv_{Sp}^{\acute{e}t}(X) \subset \text{Pre}_{Sp}(\text{Sm}_{/X}) $ admits a (symmetric monoidal) left adjoint 
\[
L_{\acute{e}t}: \text{Pre}_{Sp}(\text{Sm}_{X}) \to Shv_{Sp}^{\acute{e}t}(X)
\]

\noindent (cf. section \ref{sheaf}) which we think of as \emph{\'{e}tale sheafification}.  
Let 
 \[
\underline{K_{X}(-)}: \Catperf(X) \to \Pre_{Sp}(\text{Sm}_{X})
\]
be the functor assigning to $\text{Perf}(X)$-linear category $T \in \Catperf(X)$, the presheaf of spectra
\[
\underline{K_{X}(T)}: \text{Sm}_{X}^{op} \to Sp
\]
sending a smooth scheme $Y$ to the algebraic $K$-theory spectrum $K(\text{Perf}(Y) \otimes_{\Perf(X)} T)$.

It is not typically the case that $\underline{K_{X}(T)}$ is an \'{e}tale sheaf, for any $T \in \Catperf(X)$.  Of course, this is because $K$-theory is known to not satisfy descent with respect the \'{e}tale topology (see for example \cite[Section 1.1]{clausen}).

We can however, apply the aforementioned sheafification functor.  We compose this with the algebraic $K$-theory functor $\underline{K_{X}(-)}: \Catperf(X) \to \text{Pre}_{Sp}(\text{Sm}_{/X})$ to obtain a well defined functor which we denote by 
\[
\underline{K^{\acute{e}t}_{X}} : \Catperf(X) \to Shv_{Sp}^{\acute{e}t}(X). 
\]
Work of \cite{blumberg2} displays algebraic K-theory as a lax symmetric monoidal functor on stable $\infty$-categories; this in turn displays the functor $\underline{K_{X}(-)}:  \Catperf(X)  \to  \text{Pre}_{Sp}(\text{Sm}_{/X})$ as lax symmetric monoidal, with respect to the pointwise monoidal structure on $\text{Pre}_{Sp}(\Sm_{/X})$.  Since sheafification is itself (strongly) symmetric monoidal, the composition $\underline{K^{\acute{e}t}_{X}}$ will itself be lax symmetric monoidal.  In particular, this functor canonically factors through the forgetful functor 
\[
Mod_{{K^{\acute{e}t}_{X}}}\left(Shv_{Sp}^{\acute{e}t}(X)\right) \rightarrow Shv_{Sp}^{\acute{e}t}(X).
\]
Here we make the identification $K^{\acute{e}t}_{X}:= \underline{K^{\acute{e}t}_{X}(\mathbb{1})}$.  

Since the functor $\underline{K^{\acute{e}t}_{X}}$ is only lax symmetric monoidal, it is not immediately clear that invertible objects in $\Catperf(X)$ are sent to invertible objects in  $Mod_{K^{\acute{e}t}_{X}}\left(Shv_{Sp}^{\acute{e}t}(X)\right)$.  This is something we will need to know later.  In order to prove this, we will use the \'{e}tale local triviality result discussed in section.

\begin{prop} \label{nenu} 
The functor $\underline{K^{\acute{e}t}_{X}}: \Catperf(X) \to Mod_{K^{\acute{e}t}_{X}}\left(Shv_{Sp}^{\acute{e}t}(X)\right)$ sends invertible objects in $\Catperf(X)$  to invertible $K^{\acute{e}t}_{X}$ modules.  In particular, $\underline{K^{\acute{e}{t}}_{X}}$ induces a map of Picard spaces: 
\[
\emph{Pic}(\emph{Cat}^{\emph{perf}}(X)) \to \emph{Pic}(Mod_{K^{\acute{e}t}_{X}}\left(Shv_{Sp}^{\acute{e}t}(X)\right)).
\]

\end{prop} 

\begin{proof}

Fix $A$, an invertible object of  $\Catperf(X)$ which by proposition \ref{picbr} represents a corresponding element of the derived Brauer group $\pi_{0}\mathbf{Br}(X)$; we denote its inverse with respect to the monoidal structure by $A^{-1}$.   Next, we apply $K^{\acute{e}t}_{X}$ ; we denote the corresponding sheaves of spectra $\underline{K^{\acute{e}t}_{X}(A)}$ and $\underline{K^{\acute{e}t}_{X}(A \otimes_{\mathcal{O}_{X}} A^{-1})} \simeq \underline{K^{\acute{e}t}(\mathbb{1})}$.  Since the functor is lax symmetric monoidal, we have the following map of sheaves of spectra:
\[
\pi:  \underline{K^{\acute{e}t}(A)} \wedge_{K^{\acute{e}t}_{X}} \underline{K^{\acute{e}t}(A^{-1})} \to \underline{K^{\acute{e}t}(A \otimes_{\mathcal{O}_{X}} A^{-1})}.
\]
By Proposition \ref{localtriviality} above, Azumaya algebras are \'{e}tale locally trivial.  Hence, the \'{e}tale stalks of the corresponding sheaf of dg-categories are Morita equivalent to the base, and their corresponding \'{e}tale $K$-theory is equivalent to the \'{e}tale $K$-theory of the local ring of the stalk (which is a strictly Henselian local ring).  The functor 
\[
\phi_{x}^{*}: Shv_{Sp}^{\acute{e}t}(X) \to Sp
\]
of passing to stalks will be symmetric monoidal, as it is the left adjoint of (the stabilization of) a geometric morphism.  Hence, the stalk of $\underline{K^{\acute{e}t}(A)} \wedge_{K^{\et}_{X}} \underline{K^{\acute{e}t}(A^{-1})}$ is equivalent to the stalk of $K^{\acute{e}t}_{X}(\mathbb{1})$, namely the \'{e}tale $K$-theory of the strictly Henselian local ring of the scheme at that point.  Since the collection of stalk functors $\{\phi_{x}^{*}\}_{x \in X}$ forms a jointly conservative family of functors, this is enough to conclude that the map $\pi$ above is an equivalence.   

We have shown that $\underline{K^{\acute{e}t}_{X}(A)} \wedge \underline{K^{\acute{e}t}_{X}(A^{-1})} \simeq K^{\acute{e}t}_{X}$, the unit of the symmetric monoidal structure on $Mod_{K^{\acute{e}t}_{X}}(\text{Shv}_{Sp}^{\acute{e}t}(X))$, thereby displaying $\underline{K^{\acute{e}t}_{X}(A)}$ as an invertible object in this category.   
\end{proof}

\section{Relative Topological K-theory}
In this section we introduce our definition of relative topological $K$-theory.  We will need to use a version of topological realization defined over an arbitrary quasi-compact, quasi-separated $\C$-scheme, which we now recall.  

\subsection{Topological Realization over a varying base scheme}
Most of the following section is due to Ayoub.  The reader is instructed to check \cite{ayoub} for proofs of the statements.

Let $Y$ be a complex analytic space.  Let AnSm/$_{Y}$ denote the category of smooth complex analytic spaces  over $Y$.  We give this category a Grothendieck topology; given $Z \in $ AnSm/$_{Y}$, we let a covering of $Z$  be given by a collection of open sets $\{U_{i}\}$ that cover it in the usual topology.  These open sets will themselves be smooth analytic spaces over $Y$, as the composition of an open immersion with the structure map $f: Z \to Y$ will be smooth.  This gives the structure of a Grothendieck site on AnSm$_{Y}$.  Let $Shv_{Sp}(\text{AnSm}_{Y})$ denote the $\infty$-category of sheaves of spectra on this site, with respect to this topology.

 Let $Op(Y)$ denote the category of open subsets of $Y$.  Let $ \text{Shv}_{Sp}(Y)$ denote the $\infty$-category of sheaves of spectra on $Y$, which is obtained as the stabilization of the $\infty$ topos of sheaves of spaces over $Y$.  There is an evident inclusion  of 1-categories $\iota: Op(Y) \subset$ \text{AnSm}$_{Y}$.  Indeed, given an open subspace of $Y$, it naturally has the structure of an analytic space over $Y$ as any open subset can be endowed canonically with the structure of an analytic space, smooth over the base space.  The map of sites $Op(Y) \subset \text{AnSm}_{Y}$ induces the following adjunction between sheaf categories 
\[
\iota^{*}  : Shv_{Sp}(Y) \rightleftarrows    Shv(\text{AnSm}_{Y}) : \iota_{*}.
\]

Let  $\mathbb{D}^1 := \{z \in \C : |z| \leq 1\}$, $\mathbb{D}^{n} = (\mathbb{D}^{1})^{n}$ and let ${\mathbb{D}_{Z}^{n}} := \mathbb{D}^{n} \times Z$ for any analytic space $Z$ over $Y$.  This is an complex analytic space over $Y$.  We let Shv$^{\mathbb{D}^{1}}_{Sp}$(\text{AnSm}${Y})$ be the full subcategory of sheaves $F$ with the property that the pullback map $F(Z) \to F(\mathbb{D}^{n}_{Z})$ is an equivalence  for all smooth analytic spaces $Z$  over $Y$. Ayoub proves for any sheaf of spectra $F$, that $\iota^{*}(F) \in$ Shv$_{Sp}^{\mathbb{D}^{1}}$(\text{AnSm}${Y})$.  In fact, he proves the following:

\begin{thm}[Ayoub]\label{ayoub}
The adjunction $\iota^{*}  : Shv_{Sp}(Y) \rightleftarrows    Shv_{Sp}(\text{AnSm}_{Y}) : \iota_{*}$ induces an equivalence $Shv_{Sp}(Y) \simeq  Shv_{Sp}^{\mathbb{D}^{1}}(\text{AnSm}_{Y})$.

\end{thm}

\begin{proof}
This is found in \cite[Theoreme 1.18]{ayoub}. 
\end{proof}

Heuristically speaking, this is because the category $Shv_{Sp}(\text{AnSm}_{/Y})$ is  generated by objects of the form $\{ \mathcal{Y}(\mathbb{D}^{n}_{U}) \otimes A_{cst} \}_{n \in \mathbb{N}, U \in Op(Y)}$ where  $\mathcal{Y}(\mathbb{D}^{n}_{U})$ is the representable functor associated to $U$ and $A_{cst}$ is the constant sheaf associated to $A \in Sp$.  Passing to $\mathbb{D}^{1}$-invariant sheaves reduces the generators to those of the form $\{ U \otimes A_{cst}\}$.

Now, fix a scheme $ \phi: X\to $Spec$(\C)$. We define our realization functor on  $Shv_{Sp}^{\acute{e}t}(X)$ landing in $Shv_{Sp}(X(\C))$.  For this we must first pass to $\mathbb{A}^{1}$-invariant sheaves of spectra, which we now recall.    

\begin{defn}
Let $\mathbb{A}^{1}_{X} = \mathbb{A}^{1} \times_{Spec(\C)} X$ denote the relative affine line.  We define $Shv_{Sp}^{\acute{e}t, \mathbb{A}^1}(X)$ to be the full subcategory consisting of all  $\mathcal{F} \in Shv_{Sp}^{\acute{e}t}(X)$ such that $\mathcal{F}(Y) \to \mathcal{F}(Y \times \mathbb{A}^{1})$ is an equivalence for all $Y \in Sm_{/X}$.  As described in say \cite{dugger} (or \cite[Theorem 5.5.1.1]{lurie2009} in the setting of $\infty$-categories),  there exists a left adjoint 
\[
L_{\mathbb{A}^1}:  Shv_{Sp}^{\acute{e}t}(X) \to Shv_{Sp}^{\acute{e}t, \mathbb{A}^1}(X)
\]
to the inclusion exhibiting this as a localization.   

\end{defn}

Recall from section 4, the functor $(-)(\C): \text{Sch}/_{Spec(\C)} \to Top$ defined by 
\[
X \mapsto X(\C)
\]
where $X(\C)$ inherits the structure of a complex analytic space.  For every scheme $X$, we obtain a morphism of sites, 
\[
\text{Sm}_{X} \to \text{AnSm}_{X(\C)}
\]
which induces the adjunction 
\[
An_{X}^{*} :  Shv_{Sp}^{\acute{e}t}(X) \rightleftarrows  Shv_{Sp}(\text{AnSm}_{{X(\C)}}): An_{X,*}.
\]
If $X = Spec(\C)$, this corresponds to the functor $|| - ||_{\mathbb{S}}$ defined by Blanc. Indeed, by Theorem 3.18 of \cite{blanc}, the topological realization functor $|| - ||_{\mathbb{S}} $ is invariant upon restriction to $\Sm_{\C}$, the category of smooth schemes. Furthermore, by \cite[Theorem 3.4]{blanc}, the realization functor $||-||_{\mathbb{S}}$ factors through \'{e}tale (hyper)sheaves, and $An_{\C}^{*}$ evidently agrees with this induced functor. 

\begin{rem}

We remark that $An^{*}_{X}(\mathbb{A}_{X}^1) \simeq \mathbb{D}_{X}^1$.  Hence  $An^{*}_{X} $ sends $\mathbb{A}^{1}$-invariant sheaves to $\mathbb{D}_{X}^1$ invariant sheaves of spectra on $Shv_{Sp}(\text{AnSm}_{X(\C)})$ and therefore further descends to a functor 
\[
Shv_{Sp}^{\acute{e}t, \mathbb{A}^1}(X) \to Shv_{Sp}^{\mathbb{D}^1}(\text{AnSm}_{X(\C)}) \simeq Shv_{Sp}(X(\C)).
\]    

\end{rem}
For the remainder of this paper we will take $\widetilde{An_{X}^{*}}$ to denote the composition 
\[
Shv^{\acute{e}t}_{Sp}(X) \xrightarrow{L_{\mathbb{A}^1}} Shv_{Sp}^{\acute{e}t, \mathbb{A}^1}(X) \to Shv^{\mathbb{D}^1}_{Sp}(\text{An-Sm}/(X(\C)))    \simeq Shv_{Sp}((X(\C)) . 
\]

We will need the following fact:

\begin{lem} \label{commutinglocalizations}
There is a commutative diagram of functors 
\[
\xymatrix{
Pre_{Sp}(X)   \ar[d]_{L_{\acute{e}t}} \ar[r]^{L_{\mathbb{A}^1}} & Pre^{\mathbb{A}^1}_{Sp}(X) \ar[d]^{L_{\acute{e}t}} 
\\
Shv_{Sp}^{\acute{e}t}(X) \ar[r]_{L_{\mathbb{A}^1}} & Shv^{\acute{e}t, \mathbb{A}^1}_{Sp}(X) \\
}
\]

\end{lem}

\begin{proof}
We remind the reader that $L_{\et}$  and $L_{\Aa}$ have fully faithful right adjoints $F_{\et}$ and $L_{\Aa}$ respectively. Note that 
$F_{\et}$ preserves $\Aa$-local objects; if $\mathcal{F}$ is a $\Aa$-local,  $F_{\et}\mathcal{F}(Y) = \mathcal{F}(Y) \simeq  \mathcal{F}(Y \times \Aa) \simeq F_{\et} \mathcal{F}(Y \times \Aa)$ for any $Y \in \text{Sm}_{/X}$. From this it also follows that $L_{\et}$ preserves $\Aa$-equivalences. Indeed, if $f: \mathcal{F} \to \mathcal{G}$ was an $\Aa$-equivalence, then for any $\Aa$-local object $\mathcal{H}$,
\[
\Map(L_{\et}\mathcal{G}, \mathcal{H}) \simeq \Map(\mathcal{G}, F_{\et}\mathcal{H}) \to \Map(\mathcal{F}, F_{\et}\mathcal{H}) \simeq \Map(L_{\et}\mathcal{F}, \mathcal{H})
\]
The diagram of right adjoints commutes from the fact that $F_{\et}$ preserves $\A^1$-local objects, by the uniqueness of adjoints we conclude that the diagram in question commutes as well. 
\end{proof}

\begin{thm} \label{bae}

$\widetilde{An^{*}}( K^{\acute{e}t}_{X})$  is equivalent to the connective topological $K$-theory sheaf $\underline{ku^{X(\C)}}$ on $X(\C)$, sending an open subspace $V \subset X$ to $F(\Sigma_{+}^{\infty}V, ku)$, the complex topological K-theory spectrum of $V$.  

\end{thm}

\begin{proof}
 By \cite[Theorem 4.5]{blanc},  
\[
 \widetilde{An_{\C}^{*}}( K^{\acute{e}t}_{\C} ) \simeq  || \underline{K} ||_{\mathbb{S}} \simeq ku \in Sp
\]
where $\underline{K}$ denotes the (non-connective) algebraic $K$-theory presheaf over $Spec(\C)$ and $ku$ is the connective topological $K$-theory spectrum.  The first equivalence again follows from the fact that $An_{X}^{*} $ and  $|| - ||_{\mathbb{S}}$ agree on smooth schemes and send \'{e}tale local equivalences in $\text{Pre}_{Sp}(\text{Aff}_{\C})$ to equivalences of spectra (cf. \cite[Theorem 3.4, 3.18]{blanc}) and hence factor through \'{e}tale sheaves.  Now let $\phi : X \to Spec(\C)$ be an arbitrary quasi-compact, quasi-separated scheme.

We have the following commutative diagram of functors: 
\[
\begin{CD}
Pre_{Sp}(Sm_{\C}) @> L_{\et} >> Shv{^{\acute{e}t}_{Sp}}(\C) @>>> Shv^{\acute{e}t, \mathbb{A}^1}_{Sp}(\C) @> An^{*}>> Shv_{Sp}(*) \simeq Sp\\
@VV \phi^{*}V @VVV @VVV @VV \phi{_{an}^{*}}V\\
Pre_{Sp}(Sm_{X}) @> L_{\et} >> Shv_{Sp}^{\acute{e}t}(X)@>>> Shv_{Sp}^{\acute{e}t,\mathbb{A}^1}(X) @> An^{*} >> Shv_{Sp}(X(\C))
\end{CD}
\]
The left squares commute by the properties of localization.  The right hand square commutes simply from the formal property that the following diagram of sites
\[
\begin{CD}
\text{Sm}_{\C} @>>> \text{AnSm}_{\C} \\
@VV V @VVV \\
\text{Sm}_{X}@>>> \text{AnSm}_{X(\C)} 
\end{CD}
\]
commutes and induces a commutative diagram of  left adjoints.

\noindent Now, by Lemma \ref{commutinglocalizations}, the above diagram could have been written as
\[
\begin{CD}
Pre_{Sp}(Sm_{\C}) @> L_{\Aa} >> Pre_{Sp}^{\Aa}(Sm_{\C}) @>L_{\et}>> Shv^{\acute{e}t, \mathbb{A}^1}_{Sp}(\C) @> An^{*}>> Shv_{Sp}(*) \simeq Sp\\
@VV \phi^{*}V @VVV @VVV @VV \phi{_{an}^{*}}V\\
Pre_{Sp}(Sm_{X}) @> L_{\Aa} >> Pre_{Sp}^{\Aa}(Sm_{\C})@>L_{\et}>> Shv_{Sp}^{\acute{e}t,\mathbb{A}^1}(X) @> An^{*} >> Shv_{Sp}(X(\C))
\end{CD}
\]
By \cite{hoyoiscdh} (we can also use \cite[Th\`{e}or\'{e}me 3.9]{cisinski} together with the fact that the motivic spectrum $KGL$ represents $KH$ in $\text{SH}(X)$), there will be an equivalence $\phi^{*}(KH_{\C}) \simeq KH_{X}$, where $KH_{Y}$, as usual denotes the $\Aa$-localization of the nonconnective $K$-theory presheaf on $\text{Sm}_{Y}$. This equivalence will persist upon applying \'{e}tale sheafification and analytic realization. Hence, $\widetilde{An^{*}}( K^{\acute{e}t}_{X})$ will be equivalent to $\phi_{an}^{*}(ku)$, the constant sheaf on $X(\C)$ associated to $ku$, which we know to be the sheaf $\underline{ku^{X(\C)}}$ 
\end{proof}

\begin{rem}
By Theorem \ref{bae}, the functor $\widetilde{An^{*}}: Shv^{\acute{e}t}_{Sp}(X) \to$ Shv$_{Sp}(X(\C))$, being symmetric monoidal, induces a functor $\widetilde{An^{*}}: Mod_{K^{\acute{e}t}_{X}}(Shv^{\acute{e}t}_{Sp}(X)) \to Mod_{\underline{ku^{X(\C)}}}(Shv_{Sp}(X(\C))$.  
\end{rem}

\subsection{Relative Topological $K$-theory}
We are now in position to define relative topological $K$-theory.  Recall the fact that there exists a localization 
\[
\text{Mod}_{ku} \to \Mod_{KU} 
\] 
given precisely by $M \mapsto M \wedge_{ku} KU$.  This induces a functor at the level of $\infty$-categories of sheaves of spectra
\begin{align*}
& L_{KU}: Shv_{\Mod_{ku}}(X(\C))  \simeq \Mod_{ku}\otimes^{L} Shv_{Sp}(X)\\
& \xrightarrow{L_{KU} \otimes^{L} id }   \Mod_{KU}\otimes^{L} Shv_{Sp}(X) 
 \simeq Shv_{\Mod_{KU}}(X(\C).
\end{align*}
We can see, via the description of this functor, that the constant sheaf $\underline{ku^{X(\C)}}$ is sent to the constant sheaf $\underline{KU^{X(\C)}}$ on $X(\C)$, associated to $KU$. Indeed, the constant sheaf functor will commute with $KU$-localization.

\begin{defn}
Let $X$ be a quasi-compact, quasi-separated scheme over $\C$.   Let $T\in Cat^{\text{perf}}(X)$  be a $\text{Perf}(X)$-linear dg-category.  We let
\[
K_{X}^{top}(T): = L_{KU}(\widetilde{An_{X}^{*}}(\underline{K^{\acute{e}t}_{X}(T))})
\]

\noindent be the \emph{relative topological $K$-theory} of $T$ over $X$.

\end{defn}
We show that when $X=Spec(\C)$, this recovers the topological $K$-theory as defined in \cite{blanc}.

\begin{prop}
There is an equivalence $K^{top}_{X}(T) \simeq K^{top}(X)$ when   $X = Spec(\C)$.
\end{prop}

\begin{proof}
This is a consequence of  the fact that $An^{*}$ agrees with the restriction of $||-||_{\mathbb{S}}$ to smooth schemes and that this factors through \'{e}tale sheaves.  
\end{proof}
\subsection{Relative Topological $K$-theory as a Motivic Realization}

We give an alternate description of $K^{top}_{X}(T)$ as a \emph{motivic realization} of the dg category $T$, in the setting of stable motivic homotopy theory.  

Let $T \in \text{Cat}^{\text{perf}}(X)$. Then the nonconnective algebraic $K$-theory presheaf $\underline{K}_{X}(T)$ associated to it defined by: 
\[
Y \mapsto K(\Perf(Y) \otimes_{\mathcal{O}_{X}} T) 
\]
is a sheaf with respect to the Nisnevich topology on $\text{Sm}_{X}$.  This follows from \cite[Theorem 5.4]{luriedescent}) where it is shown that the $\widehat{\mathcal{C}at_{\infty}}$-valued presheaf on $\text{Sm}{/X}$
\[
Y \to \Perf(Y) \otimes_{\mathcal{O}_{X}} T 
\]
satisfies descent in the \'{e}tale topology  and the well known fact that non-connective $K$-theory satisfies Nisnevich descent (eg. by \cite{thomason})  Hence, we may view $\underline{K}_{X}(T)$ as an object in $Shv_{Sp}^{Nis}(X)$.

As in section 2, we let  $L_{\acute{e}t}: Shv^{Nis}_{Sp}(X) \to Shv^{\acute{e}t}_{Sp}(X)$ denote the \'{e}tale sheafication functor.  By definition, the functor 
\[
\widetilde{An}^{*} \circ  L_{\acute{e}t}   : Shv^{Nis}_{Sp}(X) \to Shv_{Sp}(X(\C)) 
\]
sends $\mathbb{A}^1$-equivalences of sheaves to equivalences in the target category.  Hence it factors through 
\[
L_{\mathbb{A}^1}: Shv^{Nis}_{Sp}(X) \to Shv^{Nis, \mathbb{A}^1}_{Sp}(X)
\]
by the universal property of $\mathbb{A}^1$ localization.  Following the conventions in \cite{cisinski}, we set $\underline{KH}_X(T):= L_{\mathbb{A}^1}(\underline{K}_{X}(T))$  to be the \emph{homotopy $K$-theory} sheaf associated to $T$. We remark that $\underline{KH}_{X}(T)$ is a module over homotopy invariant $K$-theory $\underline{KH}_{X}$ and so inherits a multiplication by the Bott  map $\beta$, where 

\[
\beta \in \pi_{0}\Map_{Shv_{Sp}^{Nis, \mathbb{A}^1}(X)} (\mathbb{P}_{X}^{1}, \underline{KH}_{X})
\] 
is the map reflecting the projective bundle theorem.  (see for example, the proof of proposition 4.3.2 in \cite{blanc}).   
\\

\noindent We pass to $\text{SH}(X)$, the stable motivic category over the scheme $X$.  By \cite[Corollary 2.22]{robalo} we may define 
\[
\text{SH}(X) = Stab_{\mathbb{P}_{X}^1}(Shv_{Sp}^{Nis, \mathbb{A}^1}(X)) := \text{lim}(Shv_{Sp}^{Nis, \mathbb{A}^1}(X) \xleftarrow{ \Omega_{\mathbb{P}^1_X } } Shv_{Sp}^{Nis, \mathbb{A}^1}(X) \ \leftarrow ...)
\]
Via the arguments in section 2.16 in \cite{cisinski} we canonically associate an object $KGL_{X}(T) \in \text{SH}(X)$ to $\underline{KH}_{X}(T)$; in effect, $KGL_{X}(T)$ will be the ``constant spectrum" with structure maps given by 
\[
\beta: \underline{KH(T)} \to \Omega_{\mathbb{P}^{1}_{X} }(\underline{KH(T)} := Map_{Shv_{Sp}^{Nis, \mathbb{A}^1}(X)}( \mathbb{P}^1_{X}, \underline{KH(T)})
\]

By (see e.g. \cite{ayoub} or \cite{robalo}) the universal property of $\text{SH}(X)$ together with theorem \ref{ayoub}, the functor $ Y \mapsto \Sigma^{\infty}_{+}(Y(\C))$  uniquely induces a functor 
\[
\Bet: \text{SH}(X) \to \text{Shv}_{Sp}(X(\C)) 
\]
We remark that $\Bet(KGL_{X}) \simeq \underline{KU^{X(\C)}}$. To see this, let $f: X \to \text{Spec}(\C)$ be the structural morphism; by \cite[Proposition 2.4]{ayoub}, $\Bet f^{*} \simeq f^{*} \mathbf{Betti}_{\C}$ and $\Bet KGL_{\C} \simeq KU$ (eg. by \cite[Section 4.6]{blanc}). It follows then, from the fact that $\Bet$ is symmetric monoiadal (\cite[Lemme 2.2]{ayoub}), that there is an induced functor 
\[
\text{Mod}_{KGL_{X}}[\text{SH}(X)] \to \text{Mod}_{\underline{KU^{X(\C)}}}(  \text{Shv}_{Sp}(X(\C)) ).
\]

We have the following proposition 

\begin{prop} \label{motives}
There is an equivalence 
\[
K^{top}_{X}(T) \simeq \Bet(KGL_{X}(T)).
\]
\end{prop}

\begin{proof}
By \cite{ayoub}, the functor $An^{*}_{X}$ sends $\mathbb{P}^{1}_{X}$ to the locally constant sheaf of spectra $S^{2}_{X} $ associated to the sphere $S^2$.  In particular, by definition of $\text{SH}(X)$ 
\[
\Bet: \text{SH}(X) \to Shv_{Sp}(X(\C))   
\]
factors through the equivalence 
\[
\Omega^{\infty}_{S^2}: Shv_{Sp^{S^2}}(X(\C)) \simeq Shv_{Sp}(X(\C)). 
\]

\noindent If we apply $\Bet$ to $KGL_{X}(T)$ we obtain the colimit of the following diagram 
\[
An^{*}_{X}(\underline{KH(T)}) \xrightarrow{\beta_{T}} \Omega^{2} ( An^{*}_{X}(\underline{KH(T)})) \xrightarrow ... 
\]
This is precisely the formula for the $L_{KU}$ localization of $An^{*}(\underline{KH(T)})$ of the $\infty$-category of $ku^{X(\C)}$ modules in $Shv_{Sp}(X(\C))$, thereby giving us the equivalence 
\[
K^{top}_{X}(T) = L_{KU}(An^{*}_{X}(\underline{K(T)}) \simeq \Bet(KGL(T))
 \]

\end{proof}
\subsection{Functoriality properties of relative topological $K$-theory}
Let $\phi: Y \to X$ be a map of schemes. We have the  following  restriction/extension adjunction :
\[
\phi^{*}: \Catperf (X) \rightleftarrows \Catperf(Y)  : \phi_{*}
\]
 where $\phi^{*}(T) \simeq T \otimes_{\Perf (X) } \Perf(Y)$  Associated to  the induced map of spaces $\phi: Y(\C) \to X(\C) $, is the adjunction  

\[
\phi^{*} : \text{Shv}_{Sp}(X (\C) \rightleftarrows \text{Shv}_{Sp}(Y (\C)) : \phi_{*} 
\]

Note that when $X = Spec(\C)$, the functor $\phi_{*}: \text{Shv}_{Sp}(X (\C)) \to \text{Shv}_{Sp}(Y (\C)) \simeq Sp$ is none other than the functor sending a sheaf of spectra to its spectrum of global sections.\\
It is immediately clear, by the properties of restriction on smooth morphisms, that $K^{top}_{X} \circ \phi^{*} \simeq \phi^{*} \circ K^{top}_{Y}$ when $\phi: X \to Spec(\C)$ is smooth.  This equivalence, via the adjunction morphisms 
\[
\mathbb{1} \to \phi_{*} \phi^{*}, \phi^{*} \phi_{*} \to \mathbb{1}
\]
gives rise to the following natural transformation    
\[
\eta: K^{top}_{X} \circ \phi_{*} \to  \phi_{*} \circ K^{top}_{Y}
\]
We do not yet know this to be an equivalence even for smooth $Y$.  This issue is particularly transparent before we apply the Bott-localization functor $L_{KU}$.  If we set  $X = Spec(\C)$, and let $Y$ be an arbitrary $\C$-scheme, then  by theorem \ref{bae},
\[
\phi_{*} \widetilde{An^{*}_{X}}(\underline{K^{\acute{e}t}_{X}(\mathbb{1})}) \simeq ku(X(\C)).
\]
Meanwhile, the semi-topological $K$-theory, $K^{st}(X)$, is typically not equivalent to $ku(X(\C))$.  Indeed $K_{*}^{st}(X)/n \simeq K_{*}(X)/n$ which is not true for connective complex $K$-theory with finite coefficients.  We suspect however that applying  Bott localization $L_{KU}$  \textbf{does eliminate} this discrepancy and plan to investigate this in future work.  The following proposition, which we will be critical for us here, serves as evidence for this hypothesis.

\begin{prop} \label{yoverx}
Let $\phi : Y \to X$ be a scheme, proper over $X$.  Let $\emph{Perf(Y)}$ be the associated $dg$-category of perfect complexes, viewed as $\emph{Perf}(X)$-linear category. Then, 
\[
K^{top}_{X}({\emph{Perf}(Y)}) \simeq \phi_{*}(K^{top}_{Y}( \mathbb{1}_{Y}))
\]  
where $\mathbb{1}_Y$ is the unit in $\emph{Cat}^{\emph{perf}}(Y)$.  

\end{prop}

\begin{proof}
By proposition \ref{motives}, 
\[
K^{top}(\text{Perf}(Y)) \simeq \Bet(KGL_{X}(\text{Perf}(Y)). 
\]
Hence, we argue using stable motivic homotopy theory.   The map of schemes $ \phi: Y \to X$   
induces the push-forward $\phi_{*} : \text{SH}(Y) \to \text{SH}(X)$ fitting in the following diagram of functors, which we claim is commutative because of the properness assumption on $\phi$:

\begin{equation}\label{Sh}
\begin{CD} 
\text{SH}(Y) @> \mathbf{Betti}_{Y}>> Shv_{Sp}(Y(\C)) \\
@VV \phi_{*}V @VV   \phi_{*}V \\
\text{SH}(X) @> \Bet >> Shv_{Sp}(X(\C))
\end{CD}
\end{equation}
To see this, we note that the functor ${Shv}^{an} :\text{Sch}_{\C} \to \text{CAlg}(\mathcal{P}r^{L, st})$ given by 
\[
X \mapsto Shv_{Sp}(X(\C))
\] 
satisfies the six-functor formalism as described in \cite[Remarque 3.3]{ayoub}. In particular,  for any arbitrary map of schemes $f : X \to Y$, there will be an adjunction 
\[
f^{*}: Shv_{Sp}(Y(\C)) \to Shv_{Sp}(X(\C)) : f_{*}.
\]
Furthermore, Betti realization  
\[
\mathbf{Betti} : \text{SH}^{\otimes} \to Shv^{an}
\]
will be a natural transformation of such functors with the property that, for $\phi: X \to Spec(\C)$ proper, 
\[
\Bet\circ \phi_{*} \simeq  \phi_{*}  \circ\mathbf{Betti}_{Y}.
\]
We remark that these facts follow from \cite[Theoreme 3.4]{ayoub} with the stronger assumption that $\phi$ is projective; the case for proper morphisms (and using $\infty$-categorical language) is stated in \cite[Proposition A.4]{blanctoen}.

Next we remark that 
\[
\phi_{*}(KGL_{Y}(\mathbb{1}_{Y})) \simeq KGL_{X}(\Perf(Y)),
\]
where the right-hand side is the object in $\text{SH}(X)$ associated to the presheaf $KH_{X}(\text{Perf}(Y))$. This just follows from the formula  for the pushforward in $\text{SH}(X)$.  To deduce the proposition, it is therefore enough to understand $\phi_{*}(\mathbf{Betti}_{Y}(KH(\mathbb{1}_{Y}))) \in Shv_{Sp}(X(\C))$.

By Theorem \ref{bae},  $K^{top}_{Y}(\mathbb{1}_{Y})= \mathbf{Betti}_{Y}(KGL_{Y})$ is the sheaf of spectra on $Y(\C)$ sending an open subset $U \subseteq Y(\C) $ to $F(U, KU)$.  Its push-forward  $\phi_{*}(K^{top}_{Y}(\mathbb{1}_{Y}))$ will be the sheaf of spectra on $X(\C)$ defined by the assignment \[
V \mapsto F(\phi ^{-1}(V), KU)
\]
for any open set $V \subseteq X(\C)$, where $\phi: Y(\C) \to X(\C)$ is the induced map on analytic spaces.  By the commutativity of (\ref{Sh}), we now conclude that 
\begin{align*}
K^{top}_{X}(\text{Perf}(Y)) & \simeq \Bet(KGL_{X}(\Perf(Y)) \\
& \simeq \phi_{*}(\mathbf{Betti}_{Y}(KGL_{Y})) \\
& \simeq \phi_{*}(\underline{KU^{Y(\C)}}) \\
& \simeq \phi_{*}(K^{top}_{Y}(\mathbb{1}_{Y})).
\end{align*}
\end{proof}  

\begin{rem}
One may, in view of the discussion preceding Proposition \ref{yoverx}, think of it as a relative version of the theorem of Friedlander-Walker that Bott-inverted semi-topological $K$-theory recovers the topological $K$-theory of $X(\C)$. Of course, they were not working in a non-commutative setup but their definition has been shown in \cite{antieauheller} to be equivalent to Blanc's for dg-categories.
\end{rem}

\section{Local Systems and Twisted Cohomology theories}

We review the version of twisted topological $K$-theory we will be working with, in its modern homotopy theoretic formulation.   Although we will be focusing on twists of $KU$, one may twist any cohomology theory $E$ in an analogous manner.  More of the general theory may be found in \cite{ando} or \cite{twists}.

\subsection{Local Systems}

We first introduce the notion of local systems  of objects of an $\infty$-category $\mathcal{C}$ on a space $X$.  The category of local systems of $KU$-module spectra will play a central role in our version of twisted topological $K$-theory.
\begin{defn}
Let $X$ be a space, (thought of as an $\infty$-groupoid) and let $\mathcal{C}$ denote an arbitrary $\infty$-category.  We define the $\infty$-category of $\mathcal{C}$ valued local systems on $X$ by Loc$_{X}(\mathcal{C}) := \text{Fun}(X, \mathcal{C})$ \\

\end{defn}  

More generally, one may define a family of objects in a presentable $\infty$-category $\mathcal{C}$ parametrized by an object $X \in \mathcal{X}$ as
\[
\text{Shv}_{\mathcal{C}}(\mathcal{X}_{/X})
\]
If $\mathcal{X}= \mathcal{S}$, we recover the category $\mathcal{S}_{/ X}$ of local systems of a space.   
\\

We make a few remarks about  Loc$_{X}(\mathcal{C})$ for general categories $\mathcal{C}$.  By the properties of taking functor categories (see, for example \cite[Proposition 1.1.3.1]{lurie2016})  if $\mathcal{C}$ is stable, then so is Loc$_{X}(\mathcal{C})$.  Furthermore, Loc$_{X}(\mathcal{C})$ will be symmetric monoidal if $\mathcal{C}$ is itself symmetric monoidal.  The unit is precisely the constant functor $\mathbf{1}: X \to \mathcal{C}$ sending every zero simplex of $X$ to the unit  $\mathbf{1}_{\mathcal{C}} \in \mathcal{C}$ and every morphism to the identity morphism of  $\mathbf{1}_{\mathcal{C}}$.   
\\

Fix an $\infty$-topos $\mathcal{X}$ and a presentable $\infty$-category $\mathcal{C}$.  Recall from section \ref{sheaf} the definition of a locally constant object in category of sheaves $\text{Shv}_{\mathcal{C}}(\mathcal{X})$.  Local systems will be significant to us in part because they, in suitable situations, give another description of locally constant objects.

We investigate this identification first when $ \mathcal{C} = \mathcal{S}$.  For now, let $\mathcal{X}$ be a general $\infty$-topos and let $\pi^{*}: \mathcal{S} \to \mathcal{X}$ be the  left adjoint to the terminal geometric morphism.  This functor preserves finite limits and therefore admits a pro-left adjoint $\pi_{!} : \mathcal{X} \to \text{Pro}(\mathcal{S})$ to the $\infty$-category of pro-spaces.  We call $\pi_{!}(\mathbf{1})$  the \textit{shape} of $\mathcal{X}$.  We say that $\mathcal{X}$ is \emph{locally of constant shape} if  $\pi_{!}$ factors through the inclusion of constant pro-spaces $\mathcal{S} \to \text{Pro}(\mathcal{S})$.  In this setting, $\pi_{!}: \mathcal{X} \to \mathcal{S}$ will be a further left adjoint to $\pi^{*}$.
\\

We now specialize to the case where $X$ is a topological space which is locally contractible.  The $\infty$-topos $Shv(X)$ will be locally of constant shape by \cite{lurie2016} ; hence in this setting we have the morphism $\pi_{!}:   Shv(X) \to \mathcal{S}$, left adjoint to $\pi^{*}$.  

We now have a canonical functor:
\[
Shv(X) \simeq Shv(X)_{/ \mathbf{1}} \xrightarrow{ \pi_{!}} \mathcal{S}_{/ \pi _{!} (\mathbf{1})}
\]
which we denote as $\psi_{!}$.  This functor admits a right adjoint, which we denote by $\psi^{*}$.  It can be described informally by :
\[
\psi^{*}(Y) =  \pi^{*}(Y)  \times_{\pi^{*} \pi_{!}(\textbf{1})} \textbf{1}.
\]

\begin{thm} \label{constantshape}
Let $\mathcal{X}$ be an $\infty$-topos which is locally of constant shape  and let $\psi^{*}: \mathcal{S}_{/ \mathcal{\pi_{!}(\textbf{1})}} \to \mathcal{X}$  be the above functor.  Then $\psi^{*}$ is a fully faithful embedding, whose essential image is the full subcategory of $\mathcal{X}$ spanned by the locally constant objects. 

\end{thm}

\begin{proof}
This is \cite[Theorem A.1.15]{lurie2016}. 
\end{proof}

If $X$ is again, locally contractible, then it satisfies the conditions of \cite[Definition A.4.15]{lurie2016} and therefore is of \emph{singular shape}; this means that we may identify $\pi_{!}(\mathbf{1})$ with the simplicial set $\text{Sing}(X)$.  

We remark further that the identification of theorem \ref{constantshape} is symmetric monoidal.  This follows from the fact that we can think of $\psi^{*}: \mathcal{S}_{/ \mathcal{\pi_{!}(\textbf{1})}} \to Shv(X)$ as a composition of functors 
\[
\pi^{*} : \mathcal{S}_{/ \pi_{!}(\textbf{1})} \to Shv(X)_{/ \pi^{*} \pi_{!}(\mathbf{1})}
\]
followed by the change of base functor 
\[
Shv(X)_{/ \pi^{*} \pi_{!}(\mathbf{1})} \to Shv(X)_{/ \mathbf{1}} \simeq Shv(X)
\]
induced by the unit of the adjunction $ \mathbf{1} \to \pi^{*} \pi_{!}(\mathbf{1})$.  Of course, each of these functors preserve cartesian products, and therefore preserve the relevant symmetric monoidal structure.    To recap, we have displayed $\psi^{*}: Loc_{Sing(X)}(\mathcal{S}) \simeq \mathcal{S}_{/ Sing(X)} \to Shv(X)$ as an algebra map in $\mathcal{P}r^{L}$.
\\

To promote this to the level of spectra, we recall from \cite[Section 1.4.4]{lurie2016} that the stabilization functor is functorial at the level of presentable $\infty$-categories; in our situation this means we can stabilize the adjunction  
\[
\psi_{!} : Shv(X) \leftrightarrow \mathcal{S}_{/ \pi_{!}(\mathbf{1})} \simeq Loc_{\pi_{!}(\mathbf{1})}(\mathcal{S}) : \psi^{*}
\]
to obtain the following adjunction of presentable stable $\infty$-categories
\[
Stab(\psi_{!}) : Shv_{Sp}(X) \leftrightarrow Stab(\mathcal{S}_{/ \pi_{!}(\mathbf{1})}) = Loc_{\pi_{!}(\mathbf{1})}(\mathcal{S}) : Stab(\psi^{*}).
\] 
It now follows from the functoriality of stabilization that this functor is fully faithful onto its image.  Indeed as endofunctors on $Loc_{X}(Sp)$, the following chain of equivalences hold 
\[
Id \simeq Stab( \psi_{!} \circ \psi^{*}) \simeq Stab(\psi_{!}) \circ Stab(\psi^{*})
\]
Furthermore, $Stab( \psi^{*})$ is symmetric monoidal.  Indeed it may be displayed as the composition of symmetric monoidal functors 
\[
Loc_{\pi_{!}(\mathbf{1})}(Sp) \xrightarrow{ Stab(\pi^{*})} Shv_{Sp}(X)_{/ \pi^{*} \pi_{!}(\bold{1})} \rightarrow  Shv_{Sp}(X)
\]
where the second map is the stabilization of the functor induced by pullback along $\bold{1} \to \pi^{*} \pi_{!}(\bold{1})$.

We may summarize the above discussion with the following proposition:

\begin{prop} \label{locallyconstant}
Let $X$ be a (locally contractible) topological space.  There exists a symmetric monoidal fully faithful right adjoint $\psi^{*}: \emph{Loc}_{X}(Sp) \to Shv_{Sp}(X)$ with essential image the locally constant sheaves of spectra.  

\begin{rem}
In the above proposition, we may replace the category $Sp$ with $\text{Mod}_{KU}$, or more generally $\text{Mod}_{E}$ for any ring spectrum $E$.   
\end{rem}
\end{prop}

\subsection{ Twisted topological K-theory}
Let $ X \in \mathcal{S}$ be a space and let $\alpha \in Loc_{X}(\text{Pic}_{KU})$ be a map $\alpha: X \to \text{Pic}_{KU}$.  We take the composition $\tilde{\alpha} : X \to \text{Pic}_{KU} \to \text{Mod}_{KU}$.  This \emph{twist} classifies a bundle  of invertible $KU$-modules over $X$.   

\begin{defn}
We define the twisted $KU$-homology to be the Thom spectrum 
\[
M\alpha = \text{colim}_{X} \tilde{\alpha }
\] 
and the twisted $KU$-cohomology to be 
\[
KU^{\alpha}(X) \simeq F_{KU}(M\alpha, KU),
\]
the internal function spectrum.  We call this the \emph{twisted K-theory}.  Alternatively, the twisted $KU$-cohomology may be defined as the spectrum of sections of this bundle i.e $\Gamma_{X}(\alpha) := Maps_{Loc_{X}(\text{Mod}_{KU})}(1_{X},  - \tilde{\alpha})$.   
\end{defn}

\begin{rem}
The reader may have noticed a difference in signs between the two definitions.  Indeed anti-equivalent in the sense that we must dualize the twist with respect to a canonical involution on the category $\text{Loc}_{X}(\text{Pic}_{E})$ before taking Thom spectra in order to show that we get the same definition for cohomology.  
 \\
 
The reader should consult \cite[Section 5]{twists} for a more complete description of the theory.   
\end{rem}

\section{Topological $K$-theory of derived Azumaya algebras} 

Having set up the relevant machinery, we restate and prove our main theorem.

\begin{thm} \label{thm:neat}
Let $X$ denote a quasi separated, quasi compact scheme  over the complex numbers.  Let $\alpha \in \pi_{0} \mathbf{Br}(X)$ be a Brauer class, and $\emph{Perf}(X, \alpha) \in \emph{Cat}^{\emph{perf}}(X)$ denote the associated $\emph{Perf(X)}$-linear category.  Then there exists a functorial equivalence
\[
K^{top}_{X}( \emph{Perf}(X, \alpha) ) \simeq \underline{KU^{\widetilde{\alpha}}(X (\C))}.
\]
Here, $ \underline{KU^{\tilde{\alpha}}(X (\C))}$ is the locally constant sheaf associated to a local system of invertible $KU$ modules; this is in turn given by a twist $\widetilde{\alpha}: X(\C) \to \text{Pic}_{KU}$ obtained functorially from $\alpha$.  
\end{thm}

We can rephrase the theorem as saying that there exists a unique lifts making the following diagram of functors commute: 
\[
\xymatrix{  
&& Loc_{\text{Sing}(X(\C))}(Pic_{KU})  \ar@{^{(}->}[d]  \\
& \mathbf{Br}(X)  \ar@{.>}[ur]   \ar[d]^{K^{\acute{e}t}_{X}(-)} \ar@{.>}[r]
& \text{Loc}_{\text{Sing}(X(\C))}(\text{Mod}_{KU}) \ar@{^{(}->} [d]_{\psi^{*}} \\
& \text{Mod}_{K^{\acute{e}t}_{X}}(Shv^{\acute{e}t}_{Sp}(X)) \ar[r]^{\widetilde{An^{*}_{X}}} & \text{Shv}_{Mod_{KU}}(X(\C)) }
\]

The existence of these lifts will follow once we show that $K^{top}_{X}( \text{Perf}(X, \alpha))$ is both a locally constant sheaf, and is invertible as an object in $\text{Shv}_{\text{Mod}_{KU}}(X(\C))$.  To this end we will use proposition \ref{nenu} (and in particular, the \'{e}tale-local triviality of derived Azumaya algebras) in an essential way.  \\

\begin{prop} \label{suh1}
Let $A \in Cat^{\text{perf}}(X)$  be $\emph{Perf}(X)$-linear category over $X$ corresponding to a derived Azumaya algebra.  Then $K_{X}^{top}(A)$ is a locally constant sheaf of $KU$-module spectra on $X(\C)$
\end{prop}

\begin{proof}
 Let  $x \in X(\C)$ be a  point.  We will show that there exists some open neigborhood  $x \in V$ for which the restriction $ K_{X}^{top}(A)|_{V}$ is equivalent to $\underline{KU^{V}}$ in $Shv_{Sp}(V)$.  The result will follow since $\underline{KU^{V}}$ is the sheafification of the constant sheaf on $V$ sending all open sets to $KU$, and is hence locally constant.

By proposition \ref{nenu}, if $A$ is a derived Azumaya algebra over $X$, representing a Brauer class $\alpha \in \pi_{0}\mathbf{Br}_{0}(X)$, its associated \'{e}tale $K$-theory theory sheaf of spectra is \'{e}tale locally equivalent to $K^{\acute{e}t}_{X}$ (the \'{e}tale $K$-theory sheaf of the base). This means that, for any closed point $x \in X$, there exists an \'{e}tale map  $\phi :Spec(S) \to X$ with image containing $x$ for which  
\[
\phi^{*}(\underline{K^{\acute{e}t} (A)})  \simeq K^{\acute{e}t}|_{Spec(S)}.
\]
Taking complex points, we obtain a map of spaces $ \widetilde{\phi}: U := Spec(S)(\C) \to X(\C)$.  Since $\widetilde{\phi} : U \to X(\C)$ is the realization of an \'{e}tale morphism it is a local homeomorphism and therefore there exists a cover $\{U_{i}\}_{i \in I}$ of $U$ such that each $U_{i}$ is mapped homeomorphically onto its image.  Choose some $U_{i}$ with $x \in \widetilde{\phi}(U_{i})$.  We now have the following chain of equivalences in $Shv_{Sp}(U_i)$:

\begin{align*}
K^{top}_{X}(A)|_{U_i} & \simeq K^{top}_{Spec(S)}(A \otimes_{\mathcal{O}_X} Mod_{S})|_{U_i} \\
& \simeq K^{top}_{Spec(S)}(\mathbb{1})|_{U_i} \\
& \simeq \underline{KU^{U}}|_{U_i} \\
& \simeq \underline{KU^{U_i}}, 
\end{align*}
where the second equivalence follows from theorem \ref{localtriviality} and the third and fourth follow from  \ref{bae}.  We have displayed, for $x \in X$, an open set $V:=U_{i}$ over which the restriction $K^{top}_{X}(A)|_{V}$ is equivalent to the constant sheaf associated to $KU$.  Hence, $   K^{top}_{X}(A)$ is itself locally constant.  
\end{proof}

\begin{rem}
Together with Proposition \ref{locallyconstant}, this allows us to identify $K^{top}_{X}(A)$ with its associated local system in $Loc_{\text{Sing}(X(\C))}(Mod_{KU})$.

\end{rem}

\begin{prop} \label{suh2}
Let $\emph{Perf}(X, \alpha)$ be as above.  Then $K^{top}_{X}(\emph{Perf}(X, \alpha))$ is an invertible object of  $Loc_{Sing(X(\C))}(\emph{Mod}_{KU})$, that is  
\[
K^{top}_{X}(\emph{Perf}(X, \alpha)) \in \emph{Pic}[\emph{Loc}_{Sing(X(\C))}(\emph{Mod}_{KU})].
\]
\end{prop}

\begin{proof}
As we showed in proposition \ref{nenu}, the associated \'{e}tale K-theory sheaf $K^{\acute{e}t}_{X}(\text{Perf}(X, \alpha)  $ on $\text{Shv}^{\acute{e}t}(X)$ is invertible as an object of $Mod_{K^{\acute{e}t}_{X}}(Shv^{\acute{e}t}_{Sp}(X))$.  Since $\mathbb{A}^{1}$-localization and the topological realization functor $An^{*}$ are symmetric monodal by \cite{ayoub}, it follows that $\widetilde{An^{*}}(K^{\acute{e}t}(A))$ is invertible as well.  Finally, the $KU$-localization functor
\[
L_{KU}: Shv_{\text{Mod}_{ku}}(X(\C)) \to  Shv_{\text{Mod}_{KU}}(X(\C))
\]
is itself symmetric monoidal; combining all this, we conclude that 

\[
K^{top}_{X}(\text{Perf}(X, \alpha)) = L_{KU}\widetilde{An^{*}_{X}}(K^{\acute{e}t}_{X}(\text{Perf}(X, \alpha)
\]
is an invertible object in $Shv_{\text{Mod}_{KU}}(X(\C))$. 
\end{proof}

\begin{rem}
The above proposition allows us to think of $K^{top}_{X}( \text{Perf(X,} \alpha)$ as an object in the Picard $\infty$-groupoid $\text{Pic}(Loc_{X(\C)}( Mod_{KU}))$.
We identify 
\[
\text{Pic}(Loc_{X(\C)}(Mod_{KU}) )\simeq  Loc_{X(\C)}(\text{Pic}_{KU})
\]
where $\text{Pic}_{KU}$ denotes the Picard space of $Mod_{KU}$.  This follows from the pointwise symmetric monoidal structure on $Loc_{Sing(X(\C))}(\text{Mod}_{KU})$: a local system $\alpha: \text{Sing}(X(\C)) \to \Mod_{KU}$ will be invertible  if and only if each simplex is sent to an invertible $KU$ module.  

\end{rem}

\begin{proof} [Proof of Theorem \ref{thm:neat}]
Let $\text{Perf}(X,\alpha)$ be the $\text{Perf}(X)$ linear category of modules over the derived Azumaya algebra associated to $\alpha \in Br(X)$. By proposition \ref{suh1}, $K^{top}_{X}(\Perf(X, \alpha))$ is locally constant; hence we may look at its image in $Loc_{\text{Sing(X)}}(\text{Mod}_{KU})$.  By proposition \ref{suh2}, $K^{top}_{X}(Perf(X, \alpha))$ is invertible as an object in the symmetric monoidal $\infty$-category  $Loc_{\text{Sing}(X(\C))}(\text{Mod}_{KU})$.  By the remarks above, we may therefore represent $K^{top}_{X}(\Perf(X, \alpha))$ by a local system $\widetilde{\alpha}: Sing(X(\C)) \to \text{Pic}_{KU} \to \text{Mod}_{KU}$.   This gives precisely the desired twist of $K$-theory.  
\end{proof}
  
\subsection{Cohomological Brauer classes} 
We now restrict to the setting where our chosen Brauer class $\alpha$ lives in $H_{\acute{e}t}^2(X, \mathbb{G}_{m})$. We show that in this case, the corresponding local system $\alpha: Sing(X(\C)) \to \text{Pic}_{KU}$ factors through the map $K( \mathbb{Z}, 3) \to Pic_{KU}$.  Hence, the twist obtained in theorem \ref{thm:neat} arises from a class  $\tilde{\alpha} \in H^{3}(X,\mathbb{Z})$.  Hence, our first order of business is to study the homotopy of the space $\text{Pic}_{KU}$. It is straightforward to see that
\[
\Omega(Pic_{KU}) \simeq GL_{1}(KU);
\]
In other words, $\text{Pic}_{KU}$ is a delooping of the space of units of $KU$.  Indeed, for a general symmetric monoidal $\infty$-category $\mathcal{C}$ with unit $\mathbf{1} \in \mathcal{C}$ and any invertible object $X \in \mathcal{C}$ , we have the following equivalences :
\[
Map_{\mathcal{C}}(X,X)
 \simeq Map_{\mathcal{C}}( \mathbf{1} \otimes X, X) \simeq Map_{\mathcal{C}}(\mathbf{1}, X^{-1} \otimes X) \simeq Map_{\mathcal{C}}( \mathbf{1}, \mathbf{1})
\]
where the second equivalence follows from the fact that we can view tensoring with $X$ as an left adjoint functor (because it is invertible, hence dualizable). Its right adjoint is none other than tensoring with its dual, $X^{-1}$.  The final equivalence holds because $X^{-1} \otimes X \simeq \mathbf{1}$.   Now, recall that when forming the $\infty$-groupoid $Pic(\mathcal{C})$, we restrict to the subcategory of equivalences.  This has the effect of restricting the endomorphism mapping spaces for each $X \in Pic(\mathcal{C})$ to the space $Aut(\mathbf{1})$ with path components corresponding to $\pi_{0}(End_{\mathcal{C}}(\mathbf{1}))^{\times} \subseteq \pi_{0}(End_{\mathcal{C}}(\mathbf{1}))$.
In this particular case, where $\mathcal{C} = Mod_{KU}$, this means that
\begin{align*}
& \pi_{1}(Pic_{KU})= \pi_{0}(GL_{1}(KU)) = \pi_{0}(\Omega^{\infty}(R))^{\times}; \\
& \pi_{n}(Pic_{KU}) \simeq \pi_{n-1}(GL_{1}(KU)) \simeq \pi_{n-1}(\Omega^{\infty}(KU)).
\end{align*}

Recall from \cite{snaith} that there is a decomposition  of infinite loop spaces: 
\[
GL_{1}(KU) \simeq K( \mathbb{Z}/2, 0) \times K( \mathbb{Z}, 2) \times BSU_{\otimes} 
\]
and therefore
\[
BGL_{1}(KU) \simeq K(\mathbb{Z}/2,1) \times K(\mathbb{Z},3) \times BBSU_{\otimes}
\]

For the remainder of this section, we work in the $\infty$-topos, $Shv^{et}(\C)$, of \'{e}tale sheaves over $Spec(\C)$.  Earlier we showed that, to a given Azumaya algebra $A$, the associated \'{e}tale $K$ theory sheaf is invertible as a module over \'{e}tale $K$-theory over the base.  This implies that there exists a map of sheaves $\mathbf{Br} \to  \widetilde{Pic}_{K^{et}}$ in $\text{Shv}^{\acute{e}t}(X)$ where we think of $\widetilde{Pic}_{K^{et}}$ as the stack of invertible modules over \'{e}tale K-theory.  In particular, the space of sections over $Spec(R) \to Spec(\C)$ is just the Picard space of the category of $K^{et}|_{R}$ modules in the \'{e}tale $\infty$-topos of $R$.  

As discussed above, the functor of taking complex points induces a morphism of topoi 
\[
|| - ||: Shv^{\acute{e}t}(\C) \to \mathcal{S}.
\]
We claim that 
\[
||\widetilde{\text{Pic}_{K^{\acute{e}t}}} || \simeq \text{Pic}_{KU}.
\]
Due to the decomposition  of spaces (for a general symmetric $\infty$-category $\mathcal{C}$),  
\[
Pic(\mathcal{C}) \simeq \pi_{0}(Pic( \mathcal{C})) \times B Aut( 1),
\]
it is enough to prove the following:

\begin{prop} \label{bgl}
There is an equivalence  of spaces $||BGL_{1}(K^{et})|| \simeq BGL_{1}(KU)$.
\end{prop}

\begin{proof}
It is enough to show that $||GL_{1}(K^{et})|| \simeq GL_{1}(ku)$.  We have a following homotopy pullback square in $Shv^{\acute{e}t}_{\C}$ defining the sheaf of units $GL_{1}(K^{\acute{e}t})$:
\[
\begin{CD}
GL_{1}(K^{\acute{e}t}) @>>> \Omega^{\infty}(K^{\acute{e}t}) \\
@VVV @VVV \\
\tau_{0}(\Omega^{\infty}(K^{\acute{e}t}))^{\times}  @>>> \tau_{0}(\Omega^{\infty}(K^{\acute{e}t}) )
\end{CD}
\]

\noindent where the bottom horizontal map is inclusion of  the grouplike components into $ \tau_{0}(\Omega^{\infty}(K^{\acute{e}t}))$.  We will identify the image, under the realization functor, $|| - ||$ of this pullback square in $\mathcal{S}$ with the one defining $GL_{1}(ku)$.  The proof will then follow since $|| - ||$ is a left adjoint to a geometric morphism of $\infty$-topoi and therefore preserves finite limits.  \\

\noindent We first show that $||\Omega^{\infty}(K^{\acute{e}t})|| \simeq \Omega^{\infty}(ku)$. Indeed 
\begin{align*}
||\Omega^{\infty}(K^{\acute{e}t})||  &\simeq || (\bigsqcup_{n\geq 0} BGL_n )^{gp}||  \\
&\simeq (\bigsqcup_{n \geq 0} ||BGL_{n}||^{gp}) \\
&\simeq ( \bigsqcup_{n \geq 0} BGL_{n}(\C))^{gp}  \\
& \simeq BU \times \mathbb{Z}.
\end{align*}
The second equivalence follows from \cite{blanc} where it is shown the the topological realization functor commutes with group completion of an $\mathbb{E}_{\infty}$ space.  (Blanc works in the context of $\Gamma$-spaces to show this.)  The third equivalence follows from the fact that  topological realization is a left adjoint and therefore commutes with coproducts.  
Next up, we show that $\Omega^{\infty}(K^{et}) \to  \tau_{0}(\Omega^{\infty}(K^{et}))$ corresponds, upon applying $|| - ||$, with the zero truncation map $\Omega^{\infty}(ku) \to \tau_{0}(\Omega^{\infty}(ku))$.  This is a consequence of \cite[Proposition 5.5.6.28]{lurie2009} where it is shown that for a colimit preserving, left exact functor $F: A \to B$ between presentable $\infty$-categories that is a left adjoint, there is a natural equivalence $ F \circ \tau_{0} \simeq \tau_{0} \circ F$.  Hence, upon applying topological realization to the right vertical arrow, we obtain the truncation map $ \Omega^{\infty}(ku) \to \tau_{0}(\Omega^{\infty}(ku))$ as desired.  \\

Finally, we show that 
\[ 
||   \tau_{0}(\Omega^{\infty}(K^{\acute{e}t})^{\times}|| \to ||\tau_{0}(\Omega^{\infty}(K^{\acute{e}t})||
\]
is the map 
\[
\tau_{0}(\Omega^{\infty}(ku))^{\times} \to  \tau_{0}(\Omega^{\infty}(ku)).
\]
For this we recall the fact that $\tau_{0}(K^{\acute{e}t})$ is the constant sheaf $\underline{\mathbb{Z}}$ and hence $\Omega^{\infty}(\tau_{0}(K^{\acute{e}t}))^{\times}$ is $\underline{\mathbb{Z}/2}$.  Let $ \pi^{*} : Shv^{\acute{e}t}(\C) \to \mathcal{S}$ be the left adjoint to the terminal geometric morphism, sending a space to its associated constant sheaf.  If we take $ || - || \circ \pi^{*}$, this is a left adjoint to geometric morphism $  \mathcal{S} \to \mathcal{S}$; of course there is only one such morphism and it will be an equivalence.  Hence, $|| \pi^{*}(A)|| \simeq A$.  Putting this all together, we conclude that $||  \tau_{0}(\Omega^{\infty}(K^{\acute{e}t})||^{\times} \to  || \tau_{0}(\Omega^{\infty}(K^{\acute{e}t})||$ is equivalent to the inclusion  $ \mathbb{Z}/ 2 \to \mathbb{Z}$ of multiplicative monoids.  
\end{proof}

\begin{cor} \label{cohomology}
Let $X$ be a quasi-compact, quasi separated scheme over $Spec(\C)$ and let $\alpha$ be a class in $H^{2}_{\acute{e}t}(X, \mathbb{G}_{m})$.   Then the assignment 
\[
\alpha \mapsto \walpha
\]
of theorem \ref{thm:neat} is given precisely by the map on cohomology 
\[
H^{2}_{\acute{e}t}(X, \mathbb{G}_{m}) \to H^{3}(X(\C),\mathbb{Z}).
\]
induced by the realization functor $ || - || : Shv_{\C}^{\acute{e}t} \to   \mathcal{S}$.

\end{cor}

\begin{proof}
By \cite{antieaugepner} there is a natural map $B^{2}\mathbb{G}_{m} \to \mathbf{Br}$.  By the discussion preceding \ref{bgl}, taking \'{e}tale $K$-theory induces the following map in $Shv_{\C}^{\acute{e}t}$.   
\[
\mathbf{Br} \to \widetilde{Pic_{K^{\acute{e}t}}}
\]
We apply the realization functor $|| - ||$ to the composition 
\[
X \xrightarrow{\alpha} B^{2}\mathbb{G}_{m} \to \mathbf{Br} \to \widetilde{Pic_{K^{\acute{e}t}}}
\] 

\noindent and obtain, by proposition \ref{bgl} the factorization of $\alpha$
\[
\text{Sing}(X(\C)) \to  K(\mathbb{Z},3) \to Pic_{KU}
\]
where the map $K(\mathbb{Z},3) \to Pic_{KU}$ is the canonical inclusion.  
\end{proof}

\subsection{}

We now restate and prove our computation of the ``absolute" $K$-theory  of Azumaya algebras.   

\begin{thm} \label{idk}
Let $X$ be a separated $\C$-scheme of finite type and let  $\alpha \in H^2_{\acute{e}t}(X, \mathbb{G}_{m})$ be a torsion class in \'{e}tale cohomology corresponding to an ordinary (non-derived) Azumaya algebra over $X$.  Then 
\[
K^{top}( \emph{Perf}(X, \alpha)) \simeq KU^{\walpha}(X(\C))
\]
where $\alpha \mapsto \walpha \in H^{3}(X, \mathbb{Z})$ is the associated cohomological twist of corollary \ref{cohomology}.    
\end{thm}

\begin{rem}
The reason \ref{idk} is non-trivial is, again, that while 
\[
\Gamma( \underline{KU^{\walpha}(X(\C))} \simeq KU^{\walpha}(X(\C)),
\]
we do not know yet that 
\[
\Gamma( K^{top}_{X}(\text{Perf}(X, \alpha)) \simeq K^{top}(\text{Perf}(X, \alpha)) 
\]
By proposition \ref{yoverx} this is true for categories of the form $\text{Perf}(Y)\in \Catperf(X)$ for $Y \to X$, smooth and proper over $X$.  We will deduce theorem \ref{idk} from this.    

\end{rem}
For the proof of this we recall the notion of a Severi-Brauer variety over a scheme $X$. One can associate, to a given cohomological Brauer class  $\alpha \in H^{2}(X, \mathbb{G}_{m})$ a scheme $f_{\alpha}: P \to X$ which is, \'{e}tale locally, equivalent to projective space.  By \cite{bernadera} there exists a \emph{semi-orthogonal decomposition} of the category 
\[
\text{Perf}(P) = \langle \text{Perf}(X, \alpha), ....\text{Perf}(X, \alpha ^{k}) \rangle
\]
where $\text{Perf}(X , \alpha)$ is the derived $\infty$-category of $\alpha$-twisted sheaves on $X$  We will not recall exactly the precise definition of a semi-orthogonal decomposition.  The significance of this for us lies in the fact that,  there will be a decomposition of the algebraic $K$-theories:

\[
K( \text{Perf}(P)) \simeq  K(\text{Perf}(X, \alpha)) \oplus .... \oplus K(\text{Perf}(X, \alpha ^{k}))
\]
This is because (connective) algebraic $K$-theory is an additive invariant and hence sends split exact sequences of stable $\infty$-categories to sums.  This will be true when we pullback to other schemes as well: if $U \in \text{Sch}_{X}$, then we have the decomposition 
\[
K(\text{Perf}( U \times_{X} P)) \simeq  K(\text{Perf}(X, \alpha \circ f)) \oplus .... \oplus K(\text{Perf}(X, \alpha ^{k} \circ f)).
\]

\begin{proof}[Proof of Theorem \ref{idk}]
Since colimits are created objectwise in the category $Shv^{\acute{e}t}_{Sp}(X)$, and since sheafification $L: Pre_{Sp}(X) \to Shv^{\acute{e}t}_{Sp}(X)$ preserves colimits, we deduce from the above that 
\[
K^{\acute{e}t}_{X}(\Perf(P)) \simeq  K^{\acute{e}t}_{X}(\Perf(X, \alpha)) \oplus .... \oplus K^{\acute{e}t}_{X}(\Perf(X, \alpha ^{k}))
\]
is a sum in $\text{Shv}_{Sp}^{\acute{e}t}(X)$.  Recall that, from here, in order to define $K^{top}_{X}(-)$, we must first $\mathbb{A}^{1}$-localize and then apply the realization functors $ An^{*}_{X} : Shv^{\mathbb{A}^1}_{Sp}(\hen) \to Shv_{Sp}(X(\C))$.  Finally we must apply the analogue of Bott inversion, $L_{KU}: \text{Mod}_{ku}[\text{Shv}_{Sp}(X( \C))] \to   \text{Mod}_{Ku}[\text{Shv}_{Sp}(X( \C))]$.  Each of these are left adjoint functors and hence preserve colimits.  It follows that 
\[
K^{top}_{X}(\Perf(P)) \simeq  K^{top}_{X}(\Perf(X, \alpha)) \oplus ... \oplus K^{top}_{X}(\Perf(X, \alpha ^{k}))
\]
is a decomposition of $K^{top}_{X}(\Perf(P))$ in $\text{Shv}_{Sp}(X(\C)$.  As taking global sections $\Gamma= \pi_{*}: Shv_{Sp}(X(\C)) \to  Sp $ is now an exact functor of stable  $\infty$-categories we obtain a corresponding decomposition of the global sections.   \\

\noindent By proposition \ref{yoverx}, 
\[
\Gamma(K^{top}_{X}(\text{Perf}(P)) \simeq K^{top}(\Perf(P))
\]
 via the associated natural transformation.  This isomorphism descends to the summands of both sides, hence we may conclude that
\[
\Gamma(K^{top}_{X}(\text{Perf}(X, \alpha)) \simeq  K^{top}( \text{Perf}(X, \alpha)).
\]    

\noindent Together with theorem \ref{thm:neat}, this allows us to finally conclude that  
\[
K^{top}(\text{Perf}(X, \alpha))) \simeq KU^{\walpha}(X (\C)),
\]
the $\walpha$-twisted topological $K$-theory of the space $X(\C)$.   
\end{proof}

\section{Applications to Projective Fiber Bundles in Topology}

We now prove our theorem on the topological $K$ theory of projective space bundles.  For the reader's convenience, we restate the theorem.  

\begin{thm}\label{finCW}
Let $X$ be a finite CW-complex.  Let $\pi: P \to X$ be a bundle of rank $n-1$ projective spaces classified by a map $ \pi: X \to BPGL_{n}(\C)$ and let $\tilde{\alpha}: X \to B^{2} \mathbb{G}_{m}$ be the composition of this map along the map $ BPGL_{n}(\C) \to B^{2} \mathbb{G}_{m} \simeq K( \mathbb{Z}, 3)$  This gives rise to an element $\tilde{\alpha} \in H^3( X, \mathbb{Z})$ which we can use to define the twisted $K$-theory spectrum $KU^{\alpha}(X)$.  Then there exits the following decomposition of spectra:

\[
KU^{*}(P) \simeq KU^{*}(X) \oplus KU^{\widetilde{\alpha}}(X) ... \oplus KU^{\widetilde{\alpha^{n-1}}}(X).
\]    
where $KU^{\widetilde{\alpha ^k}}(X)$ denotes the twisted $K$-theory with respect to the class $\widetilde{\alpha ^{k}} \in H^{3}(X, \mathbb{Z})$.

\end{thm}  

We use our results together with certain approximations of classifying spaces  by algebraic varieties, due to Totaro.  This allows us to reduce the theorem to an algebro-geometric setting, to which our previous results apply.  We begin with the following lemma:

\begin{lem}
There is a weak homotopy equivalence 
\[BPGL_{n}(\C) \simeq colim_{i \to \infty}Y^{i}(\C)
\]
where $BPGL_{n}(\C)$ denotes the classifying space of the projective linear group and each $Y^{i}(\C)$ is the space of complex points of a smooth quasi projective variety over the complex numbers.  
\end{lem}

\begin{proof}
This is a reformulation of remark 1.4 of \cite{totaro}. Let $W$ be a faithful representation of $PGL_{n}$; for instance we may choose the adjoint representation which is well known to be faithful.  Then we let 
\[
V^{N} = Hom(\mathbb{C}^{N + n}, W) \simeq W^{N + n}
\]
and let $S$ be the closed subset in $V^{N}$ of non-surjective linear maps.  The group acts freely outside of $S$.  Furthermore, the codimension of $S$ goes to infinity as $N$ goes to infinity.  Now, $Y^{N} : =(V^{N} \setminus S) / PGL_{N} $ exists as a smooth quasi-projective variety by \cite[Remark 1.4]{totaro}  We may define maps 

\[
V^{N} \to V^{N+1}
\]       
by extending the linear maps in the obvious manner; it is clear that nonsurjective maps will be sent to nonsurjective maps. Moreover these maps will be equivariant  with respect to the action.    
The colimit of this diagram is independent of the choice of faithful representation; moreover, it is shown in \cite{morel} to be $\mathbb{A}^{1}$- equivalent to the \'{e}tale classifying space of $PGL_{n}$.  Applying the induced functor $|| - ||: H_{\C} \to \mathcal{S}$, we obtain the desired colimit diagram in spaces.  
\end{proof}

\begin{proof}[Proof of Theorem \ref{finCW}]
Let $X$ be a finite CW complex, as in the statement of the theorem, equipped with a map $\alpha: X \to BPGL_{n}(\C)$.  By the above lemma, $BPGL_{n}(\C) \simeq \text{hocolim}_{i \to \infty} Y^{i}(\C)$.  We may view $X$ as a compact object in the homotopy category of spaces and therefore $ \alpha: X \to    BPGL_{n}(\C)$ factors through some $Y^{i}(\C) \to  BPGL_{n}(\C)$.  If we write this factorization as $\alpha =  \beta \circ f$ for maps $f : X \to Y^{i}(\C)$ and $ \beta: Y^{i}(\C) \to BPGL_{n}(\C)$ we see that the projective bundle $P$ can be expressed as $f^{*}(P^{i})$ the pullback of a projective space bundle over $Y^{i}$ classified by the map $\beta: Y^{i}(\C) \to BPGL_{n}(\C)$.  It is important to note that, by the above lemma, the map $\beta$ is in the image of the realization functor and therefore arises  from some map of \'{e}tale sheaves $ \beta^{alg}: Y^{i} \to BPGL_{n}$, giving rise to a Severi-Brauer scheme over $Y^{i}$.  The projective space bundle  over $Y^{i}(\C)$ can therefore be thought of as the space of complex points of this Severi-Brauer scheme.  
\\

The composition $ Y^{i}(\C) \xrightarrow{\beta} BPGL_{n}(\C) \xrightarrow{\iota} K(\mathbb{Z},3)$  along the morphism gives rise to  local system of invertible $KU$-module spectra over $Y^{i}(\C)$.  We obtain the following decomposition of local systems in $\text{Loc}_{Y^{i}(\C)}(\text{Mod}_{KU})$:

\[
\underline{KU(P^{i})} \simeq \underline{KU(Y^{i}(\C))}  \oplus \underline{KU^{\alpha}(Y^{i}(\C))} \oplus ... \oplus \underline{KU^{\alpha^{n-1}}(Y^{i}(\C))}.
\]

\noindent by the results of the previous section.  We seek to pull back this decomposition to   $Loc_{X}(\text{Mod}_{KU})$.  This is straightforward; the map $f: X \to Y^{i}(\C)$ induces an adjunction 
\[
f^{*}: Loc_{Y^{i}(\C)}(\Mod_{KU} ) \leftrightarrow Loc_{X}(\Mod_{KU}) : f_{*}
\]
The functor $f^{*}$, being a left adjoint, preserves coproducts hence giving us a decomposition of local systems
\begin{equation} \label{pullback}
f^{*}(\underline{KU(P^{i})}) \simeq f^{*}(\underline{KU(Y^{i}(\C))})  \oplus f^{*}(\underline{KU^{\alpha}(Y^{i}(\C)) })\oplus ... \oplus f^{*}(\underline{KU^{\alpha^{n-1}}(Y^{i}(\C))})
\end{equation}
 on $X$.  It remains to identify these pullbacks with the corresponding local systems on $X$.  To see that $f^{*}(\underline{KU(P^{i})}) \simeq \underline{KU(P)}$ on $X$, note that the local system $\underline{KU(P^{i})}$ is given by 
 \[
 Y \to \mathcal{S} \xrightarrow{KU^{(-)}} \text{Mod}_{KU};
\]
we obtain $f^{*}(\underline{KU(P^{i})})$ by precomposing this with $X \xrightarrow{f} Y$.  Of course the composition $X \xrightarrow{f} Y \to \mathcal{S}$ classifies the pull-back
\[
P \simeq P^{i} \times_{Y} X 
\]
as a local system of spaces over $X$; composing with $\mathcal{S} \xrightarrow{KU^{(-)}} \text{Mod}_{KU}$ gives us $\underline{KU^{*}(P)}$.  The identification of the terms on the right hand side of \ref{pullback} is immediate as the pullback will be given by the following map obtained through precomposition by $ X \xrightarrow{f} Y^{i}(\C)$: 
\[
X \xrightarrow{f} Y^{i}(\C) \xrightarrow{ \beta}  BPGL_{n}(\C) \to K(\mathbb{Z}, 3) \to Pic_{KU} \hookrightarrow \text{Mod}_{KU}
\]
This is none other than the local system corresponding to $\alpha: X \to  BPGL_{n}(\C)$.  The result is now immediate upon taking global sections.  
\end{proof}
\bibliographystyle{amsalpha}
\bibliography{biblio}
\end{document}